\documentclass[preprint,3p]{amsart}
\usepackage{wrapfig}
\usepackage{amssymb}

\usepackage{float}

\usepackage{amsmath,amsfonts,amsthm,amstext}
\usepackage[latin1]{inputenc}
\usepackage{fancybox}
\usepackage{niceframe}
\usepackage{array}
\usepackage{newlfont}
\usepackage{verbatim}
\usepackage[dvips]{psfrag}
\usepackage{color}
\usepackage{float}
\usepackage[scriptsize]{caption}
\usepackage{indentfirst}
\usepackage[normalem]{ulem}
\usepackage{pst-text}
\usepackage{graphicx}
\usepackage{graphics}
\usepackage{overpic}

\graphicspath{{./Figures/}}
\usepackage{wrapfig}

\newcommand{\s}{\Sigma}

\newcommand{\R}{\mathbb{R}}

\newcommand{\Cr}{\mathcal{C}^{r}}
\newcommand{\Xr}{\chi^{r}}
\newcommand{\Or}{\Omega^{r}}
\newcommand{\rn}[1]{\mathbb{R}^{#1}}
\newcommand{\e}{\varepsilon}
\newcommand{\V}{\mathcal{V}}
\newcommand{\W}{\mathcal{W}}
\newcommand{\U}{\mathcal{U}}

\newcommand{\sgn}{\textrm{sgn}}

\theoremstyle{definition}

\newtheorem {theorem} {Theorem} 
\newtheorem {prop} {Proposition}

\newtheorem {lem} {Lemma}

\newtheorem {defn} {Definition}
\newtheorem {rem} {Remark}

\newtheorem*{thmA}{Theorem A}
\newtheorem*{thmB}{Theorem B}

\definecolor{verde}{rgb}{0.0,0.5,0.0}
\definecolor{azul}{rgb}{0,0,128}
\definecolor{roxo}{rgb}{0.44,0.16,0.39}
\definecolor{vinho}{rgb}{0.5,0.0,0.13}
\definecolor{lilas1}{rgb}{0.6,0.33,0.73}
\definecolor{rosa}{rgb}{0.84,0.04,0.33}
\definecolor{mostarda}{rgb}{0.91,0.41,0.17}
\definecolor{mostarda2}{rgb}{1.0,0.66,0.07}

\begin{document}

\title[On Structural Stability of 3D Filippov Systems]{On Structural Stability of 3D Filippov Systems: A Semi-Local Approach}

\author{Ot\'avio M. L. Gomide}
\address[OMLG]{Department of Mathematics, Unicamp, IMECC\\ Campinas-SP, 13083-970, Brazil}
\email{otaviomleandro@hotmail.com}

\author{Marco A. Teixeira}
\address[MAT]{Department of Mathematics, Unicamp, IMECC\\ Campinas-SP, 13083-970, Brazil}
\email{teixeira@ime.unicamp.br}

\maketitle

\begin{abstract}
The main purpose of this work is to provide a non-local approach to study aspects of structural stability of $3D$ Filippov systems. We introduce a notion of semi-local structural stability which detects when a piecewise smooth vector field is robust around the whole switching manifold and we give a complete characterization of such systems. In particular, we present some methods in the qualitative theory of piecewise smooth vector fields, emphasizing a geometrical analysis of the foliations generated by their orbits. Such approach displays surprisingly a rich dynamical behaviour that has been studied in details in this work.\newline\indent
It is worth to say that this subject has not been treated recently from a non-local point of view, and we hope that the approach adopted in this work contributes to the understanding of the structural stability for piecewise-smooth vector fields in its most global sense. 	
\end{abstract}

\section{Introduction}
\label{intro}

Piecewise Smooth Vector Fields (PSVF for short) are frequently encountered in many areas of science (see for instance \cite{Ba,DiB,Br}), where phenomena are described by smooth relationships among variables but they present a different nature in some regions of the state space. Their use in mathematical modelling has being considered a natural way to obtain more accurate prediction results. Based on the versatility of applications of this field, we are encouraged to develop a well-established theory to deal with this kind of systems. 

In the classical theory of smooth vector fields, the structural stability concept determines the robustness of a model with respect to the initial conditions and parameters as well as its efficiency. From our point of view, it is of the main importance to establish this concept to PSVF in a systematic way.

In attempt to reach this goal, many papers have emerged with the purpose of the characterization of the structural stability for PSVF. In dimension $2$, the concept of local structural stability was extensively studied in \cite{GTS,K,KRG}. In \cite{BPS}, Broucke et al have studied the problem in dimension $2$ from a global point of view. In dimension $3$, the local approach has been completely characterized from papers \cite{CJ1,CJ2,GT,ST,T2}. In higher dimension, some models were treated in \cite{CJn}, but it is still poorly understood, even locally.

As far as the authors know, in dimension $3$, non-local aspects of structural stability of PSVF have not been studied yet, maybe due to its high complexity. In light of this, we introduce in this work a concept of semi-local structural stability, in order to understand what happens around the whole switching manifold (not only point-wisely) of a robust PSVF. We attempt to provide all results in the most rigorous way by considering the problem from a geometric-topological point of view.

We consider piecewise-smooth vector fields $Z$ defined in $\rn{3}$ having a compact switching manifold $\s$, and we denote this set by $\Or$. Roughly speaking, $Z_0\in\Or$ is semi-local structurally stable at $\s$ if all systems $Z\in\Or$ sufficiently near from $Z_0$ present the same behavior as $Z_0$ in a neighborhood $V\subset\rn{3}$ of $\s$. In this work, we completely characterize all the semi-local structurally stable systems at $\s$, and we also conclude that it is not a generic property in $\Or$. Also, a version of Peixoto's Theorem for sliding vector fields is obtained.

It is worth to mention that the characterization of structural stability of $3D$ PSVF in its most global comprehensive notion is one of the most complex and intriguing topic in the theory of PSVF. The semi-local approach studied in this work allows us to find constraints in this characterizing problem, and we hope that it can be allied to the study of global connections between points of $\s$ (see \cite{CPKO}, for example) as a guideline to solve this problem. 

Our paper is structured as follows. An overview of basic concepts and $3D$ tangential singularities of codimension $0$ is given in Section 2. The topological orbital equivalences used throughout this work are described in Section $3$. Section $4$ presents a formal language to deal with this problem. In Section $5$ the main results are presented. Sections $6$, $7$ and $8$ are devoted to prove the principal results. In Section $9$, we discuss some future directions of this work.

\section{Preliminaries}

In what follows we present an overall description of some basic concepts and results which will be used throughout this work.

\subsection{Filippov Systems}

Let $M=\rn{3}$ and let $f:M\rightarrow \mathbb{R}$ be a smooth function having $0$ as a regular value. Suppose that $\Sigma=f^{-1}(0)$ is an embedded codimension $1$ submanifold of $M$. Also, assume that $\s$ is compact, connected and simply connected. Consider $M^{+}=f^{-1}\left([0,+\infty)\right)$ and $M^{-}=f^{-1}\left((-\infty,0]\right)$.

\begin{defn}A \textbf{$\Cr$-germ of piecewise-smooth vector field} at $\s$ is an equivalence class  $\widetilde{Z}=(\widetilde{X},\widetilde{Y})$ of pairwise $\Cr$ vector fields defined as follows: $Z_{1}=(X_{1},Y_{1})$ and $Z_{2}=(X_{2},Y_{2})$ are in the same equivalence class if and only if:
	\begin{itemize}
		\item[(a)] $X_{i}$ and $Y_{i}$ are defined in neighborhoods $U_{i}$ and $V_{i}$ of $\s$ in $M$, respectively ($i=1,2$);
		\item[(b)] There exist neighborhoods $U_{3}$ and $V_{3}$ of $\s$ in $M$ such that $U_{3}\subset U_{1}\cap U_{2}$ and $V_{3}\subset V_{1}\cap V_{2}$;
		\item[(c)] $X_{1}|_{U_{3}\cap \overline{M^{+}}}=X_{2}|_{U_{3}\cap\overline{M^{+}}}$ and $Y_{1}|_{V_{3}\cap \overline{M^{-}}}=Y_{2}|_{V_{3}\cap\overline{M^{-}}}$.
	\end{itemize}
	In this case, $Z=(X,Y)$ is a representative of the class $\widetilde{Z}$.  
	The set of all $\Cr$-germs of piecewise-smooth vector fields at $\s$ will be denoted by $\Or$. 
\end{defn}

\begin{rem}
	In order to simplify the notation, a $\Cr$-germ of vector field $\widetilde{Z}=(\widetilde{X},\widetilde{Y})$ will be referred simply by its representative $Z=(X,Y)$.
\end{rem}

The space of all $\Cr$-germs of piecewise-smooth vector fields $Z=(X,Y)$ at $\s$ such that $X=Y$ will be denoted by $\Xr$ and it coincides with the classical space of germs of smooth vector fields of class $\Cr$ at $\s$.

Notice that $\Or=\Xr\times \Xr$. The $\Cr$ topology is introduced in the space of germs $\Xr$, and we endow $\Or$ with the product topology. 

If $Z=(X,Y)$ is a representative of the germ $\widetilde{Z}$ then, a \textbf{piecewise smooth vector field} is defined in some neighborhood $V$ of $\s$ in $M$ as follows:

\begin{equation}
Z(p)=\left\{ \begin{array}{cc}
X(p), &\textrm{if } p\in M^{+}\cap V,\\
Y(p), &\textrm{if } p\in M^{-}\cap V.
\end{array}\right.
\end{equation}

\begin{defn}
	The \textbf{Lie derivative} of $f$ in the direction of the vector field $X\in \chi^{r}$ at $p\in\Sigma$ is defined as $Xf(p)=X(p)\cdot \nabla f(p)$, and the successive Lie derivatives are given by $X^{n}f(p)=X(p)\cdot \nabla X^{n-1}f(p)$. The \textbf{tangency set} of $X$ with $\Sigma$ is given by $S_{X}=\{p\in\Sigma;\ Xf(p)=0\}$. 
\end{defn}

\begin{rem}
	Notice that the Lie derivative is well-defined for a germ $\widetilde{X}\in\Xr$ since all the elements in this class coincide in $\s$.
\end{rem}

\begin{defn}
	Let $Z=(X,Y)\in\Or$, a point $p\in \s$ is said to be a \textbf{tangential singularity} of $Z$ if $Xf(p)Yf(p)=0$ and $X(p), Y(p)\neq 0$.
\end{defn}	

If $Z=(X,Y)\in\Or$, then the tangency set of $Z$ is given by $S_Z=S_X\cup S_Y$, and the switching manifold $\Sigma$ splits into:

\begin{itemize}
	\item Crossing Region: $\Sigma^{c}(Z)=\{p\in \Sigma;\ Xf(p)Yf(p)>0\};$
	\item Stable Sliding Region: $\Sigma^{ss}(Z)=\{p\in \Sigma;\ Xf(p)<0 \textrm{ e }Yf(p)>0\};$
	\item Unstable Sliding Region: $\Sigma^{us}(Z)=\{p\in \Sigma;\ Xf(p)>0 \textrm{ e }Yf(p)<0\}.$			
\end{itemize}	

Define the \textbf{sliding region} of $Z$ as $\s^{s}(Z)=\s^{ss}(Z)\cup \s^{us}(Z)$.

\begin{rem}
	If there is no misunderstanding, the dependence of these regions on $Z$ is omitted. In addition, $\s$ may be denoted by $\s(Z)$, in order to distinguish the regions of $\s$ corresponding to $Z$, when necessary. 
\end{rem}

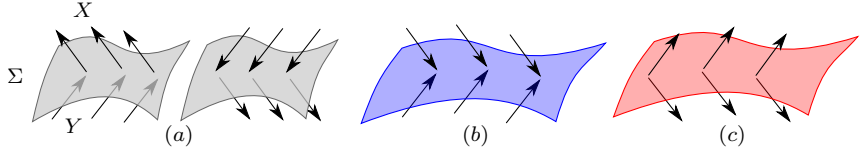
\begin{figure}[H]
	\centering
	\bigskip
	\begin{overpic}[width=11cm]{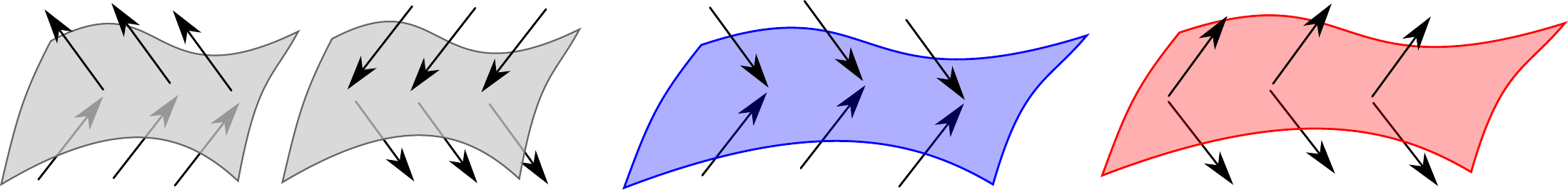}
		\put(16,-2){{\footnotesize $(a)$}}		
		\put(52,-2){{\footnotesize $(b)$}}		
		\put(83,-2){{\footnotesize $(c)$}}	
		\put(-3,5){{\footnotesize $\s$}}	
		\put(5,13){{\footnotesize $X$}}	
		\put(4,-1){{\footnotesize $Y$}}							
		
	\end{overpic}
	\bigskip
	\caption{Regions in $\s$: $\s^{c}$ in $(a)$, $\s^{ss}$ in $(b)$ and $\s^{us}$ in $(c)$.   }	
\end{figure}  	

Notice that $\s$ is the disjoint union $\s^{c}\cup \s^{s}\cup S_{Z}$. If $p$ is a point of $\s^{c}$, then $X(p)$ and $Y(p)$ points towards the same direction, in this case the solution is given by the concatenation of the orbits of $X$ and $Y$ at $p$.
Special attention must be paid to the solutions of $Z$ passing through points of $\s^s$.

Actually, there are several different ways to define it as we can see in \cite{BSFJ,F} and references therein. In this work, we follow Filippov's convention to define a local solution of $Z$. 

\begin{defn}
	If $Z=(X,Y)\in \Or$ and $p\in \Sigma^{s}$, then we define the \textbf{sliding vector field}:	
	\begin{equation}
	F_{Z}(p)=\frac{1}{Yf(p)-Xf(p)}\left(Yf(p)X(p)-Xf(p)Y(p)\right).
	\end{equation}
\end{defn}

\begin{rem}
	Notice that $F_{Z}$ is a vector field tangent to $\s^{s}$. The critical points of $F_{Z}$ in $\s^{s}$ are called \textbf{pseudo-equilibria} of $Z$.
\end{rem}

\begin{defn}
	Let $Z=(X,Y)\in\Or$, a point $p\in \s$ is said to be a \textbf{$\s$-singularity} of $Z$ if $p$ is either a critical point of $X$ or $Y$ or a tangential singularity or a pseudo-equilibrium of $Z$. Otherwise, it is said to be a \textbf{regular-regular} point of $Z$
\end{defn}

\begin{defn}\label{NSVF}
	If $p\in\s^{s}$, the \textbf{normalized sliding vector field} of $Z=(X,Y)$ is the $\Cr$ vector field defined on $\s$ by:
	\begin{equation}
	F_{Z}^{N}(p)=Y f(p)X (p)-X f(p)Y(p).
	\end{equation}	
\end{defn}

\begin{rem}
	The normalized siding vector field can be $\Cr$ extended beyond the boundary of $\s^s$. In addition, if $R$ is a connected component of $\s^{ss}$, then $F_{Z}^{N}$ is a re-parameterization of $F_{Z}$ in $R$, and so they have exactly the same phase portrait. If $R$ is a connected component of $\s^{us}$, then $F_{Z}^{N}$ is a (negative) re-parameterization of $F_{Z}$ in $R$, then they have the same phase portrait, but the orbits are oriented in opposite direction.
\end{rem}

If $p\in\s^c$, then the orbit of $Z=(X,Y)\in\Or$ at $p$ is defined as the concatenation of the orbits of $X$ and $Y$ at $p$. Nevertheless, if $p\in\s\setminus\s^c$, then there is a lack of uniqueness of solutions. In this case,  the flow of $Z$ is multivalued and any possible trajectory passing through $p$ originated by the orbits of $X$, $Y$ and $F_Z$ is considered as a solution of $Z$. More details can be found in \cite{F,GTS}.

Roughly speaking, as we are interested in studying structural stability in $\Or$ it is imperative to take into account all the leaves of the foliation in $M$ generated by the orbits of $Z=(X,Y)$ (orbits of $X$, $Y$ and $F_{Z}$). For more details see Section \ref{local}.

In a PSVF, if only one component of $Z=(X,Y)$ is considered, say $X$, then it is a germ of  $\Cr$ vector field defined on a manifold with boundary $\overline{M^{+}}$. Therefore, the theory of vector fields on manifolds with boundary (see \cite{PP,ST,T1,V}) is used to distinguish some points of $\s$.

In the next topics, some types of contact of $X$ and $Y$ with $\s$ are presented such as their geometrical interpretation.

\subsection{Tangential Singularities of Codimension $0$}

Denote the space of germs of $\Cr$ vector fields defined on the manifold with boundary $N$ by $\Xr(N)$ ($r>1$). If $N$ is not specified, then consider $N=\overline{M^+}$ or $N=\overline{M^-}$.

\begin{defn}A point $p\in \Sigma$ is said to be a \textbf{fold} point of $X\in\chi^{r}(\overline{M^{+}})$ if $Xf(p)=0$ and $X^{2}f(p)\neq 0$. If $X^{2}f(p)>0$ (resp. $X^{2}f(p)<0$), then $p$ is a \textbf{visible fold} (resp. \textbf{invisible fold}).
\end{defn}

\begin{rem}
	If $X\in\chi^{r}(\overline{M^{-}})$, the visibility condition is switched.
\end{rem}

\begin{defn}
	A point $p\in \Sigma$ is said to be a \textbf{cusp} of $X\in\Xr(N)$ if $Xf(p)=X^{2}f(p)=0$, $X^{3}f(p)\neq 0$ and $\{df(p),\ dXf(p),\ dX^{2}f(p)\}$ is a linearly independent set.
\end{defn}

Generically, a fold point of $X$ belongs to a local curve of fold points of $X$ with the same visibility, and cusp points occur as isolated points located at the extremes of curves of fold points.

\begin{defn}
	$X\in \Xr(N)$ is said to be \textbf{simple} if either $S_{X}=\emptyset$ or $S_{X}$ is just composed by fold and cusp points of $X$. The set of all simple germs of $\Xr(N)$ will be denoted by $\Xr_{S}$.
\end{defn}

In \cite{V}, S. M. Vishik used tools from Theory of Singularities to obtain sharpen results on vector fields near the boundary of an $n$-manifold. In particular, when $n=3$, the following result is stated.

\begin{theorem}[Vishik's Normal Form] \label{Vishik}
	Let $X\in\Xr_{S}$. If $p\in S_{X}$ then there exist a neighborhood $V(p)$ of $p$ in $M$, a system of coordinates $(x_{1},x_{2},x_{3})$ at $p$ defined in $V(p)$ ($x_{i}(p)=0$, $i=1,2,3$) and an integer $k=k(p)$, $k=1,2$, such that:
	\begin{enumerate}
		\item If $p$ is a fold point, then $k=1$ and $X|_{V(p)}$ is a germ at $V(p)\cap \s$ of the vector field given by:
		\begin{equation}\label{foldV}
		\left\{ \begin{array}{l}
		\dot{x_{1}}=x_{2},\\
		\dot{x_{2}}=1,\\
		\dot{x_{3}}=0.\\
		\end{array}\right.
		\end{equation}
		\item If $p$ is a cusp point, then $k=2$ and $X|_{V(p)}$ is a germ at $V(p)\cap \s$ of the vector field given by:
		\begin{equation}\label{cuspV}
		\left\{ \begin{array}{l}
		\dot{x_{1}}=x_{2},\\
		\dot{x_{2}}=x_{3},\\
		\dot{x_{3}}=1.\\
		\end{array}\right.
		\end{equation}	
		\item $\s$ is given by the equation $x_{1}=0$ in $V(p)$.
	\end{enumerate}
	
	The set $\Xr_{S}$ is open and dense in $\Xr(N)$.
\end{theorem}

\begin{rem}
	If we perform the change of coordinates $y_1=x_3$, $y_2=x_3^2-2x_2$, and $y_3=2x_1$, then system \eqref{cuspV} is carried to the system $\dot{y_{1}}=1,\ \dot{y_{2}}=0,\ \dot{y_{3}}=y_1^2-y_2,$. Analogously, if we consider the change $y_1=x_2$, $y_2=x_3$, and $y_3=x_1$, then \eqref{foldV} is carried to $\dot{y_{1}}=1$, $\dot{y_{2}}=0$, $\dot{y_{3}}=y_1$. In both cases $\s$ is given by the equation $y_3=0$. 				
\end{rem}

In the piecewise-smooth context, we consider the following tangential singularities:

\begin{defn}
	Let $Z=(X,Y)\in\Or$. A tangential singularity $p\in \s$ is said to be \textbf{elementary} if it satisfies one of the following conditions:
	\begin{itemize}
		\item[(FR) -] $Xf(p)= 0$, $X^{2}f(p)\neq 0$ and $Yf(p)\neq 0$ (resp. $Xf(p)\neq 0$, $Yf(p)=0$ and $Y^{2}f(p)\neq 0$). In this case, $p$ is said to be a \textbf{fold-regular} (resp. regular-fold) point of $\s$.	
		\item[(CR) -] $Xf(p)= 0$, $X^{2}f(p)=0$, $X^{3}f(p)\neq 0$ and $Yf(p)\neq 0$ (resp. $Xf(p)\neq 0$, $Yf(p)=0$, $Y^{2}f(p)= 0$ and $Y^{3}f(p)\neq 0$), and $\{df(p),dXf(p),dX^{2}f(p)\}$ (resp. $\{df(p),dYf(p),dY^{2}f(p)\}$) is a linearly independent set. In this case, $p$ is said to be a \textbf{cusp-regular} (resp. regular-cusp) point of $\s$.	
		\item[(FF) -] If $Xf(p)= 0$, $X^{2}f(p)\neq0, Yf(p)= 0$, $Y^{2}f(p)\neq 0$ and $S_{X}\pitchfork S_{Y}$ at $p$. In this case, $p$ is said to be a \textbf{fold-fold} point of $\s$. 
	\end{itemize}	
\end{defn}

\begin{figure}[H]
	\centering
	\bigskip
	\begin{overpic}[width=10cm]{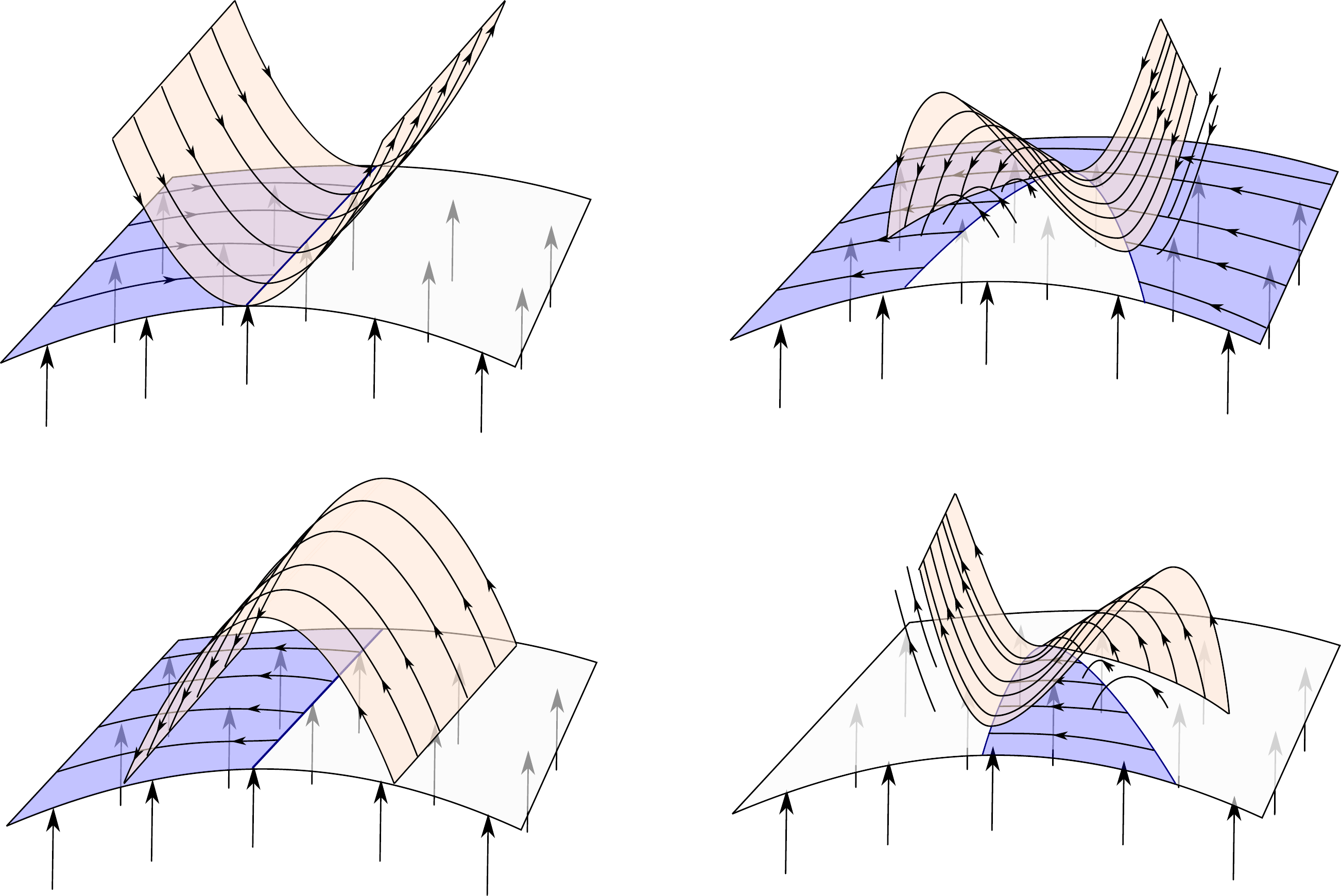}
		\put(20,34){{\footnotesize (a-i)}}
		\put(20,-2){{\footnotesize (a-ii)}}		
		\put(75,34){{\footnotesize (b-i)}}
		\put(75,-2){{\footnotesize (b-ii)}}													
		\put(0,45){{\footnotesize $\s$}}
		\put(1,53){{\footnotesize $X$}}				
		\put(0,35){{\footnotesize $Y$}}																			
	\end{overpic}
	\bigskip
	\caption{(a) Fold-Regular singularities ((i) visible and (ii) invisible) and (b) Cusp-Regular singularities ((i) $X^3f(p)<0$ and (ii) $X^3f(p)>0$).}	
\end{figure}
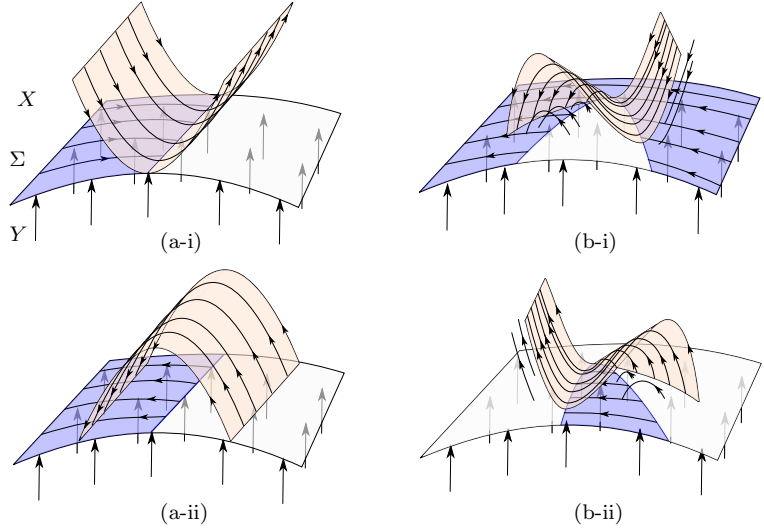 

\begin{rem}
	If $p$ is a fold-fold point of $Z=(X,Y)\in\Or$, then we classify its visibility as follows:
	\begin{enumerate}
		\item[(H) $-$] If $X^{2}f(p)>0$ and $Y^{2}f(p)<0$ then $p$ is a \textbf{visible fold-fold} point ;
		\item[(P) $-$] If $X^{2}f(p)<0$ and $Y^{2}f(p)<0$ then $p$ is a \textbf{invisible-visible fold-fold} point and if $X^{2}f(p)>0$ and $Y^{2}f(p)>0$ then $p$ is a \textbf{visible-invisible fold-fold} point;		
		\item[(E) $-$] If $X^{2}f(p)<0$ and $Y^{2}f(p)>0$ then $p$ is a \textbf{invisible fold-fold} point.	
	\end{enumerate}
	A fold-fold point of type (H), (P) and (E) is referred as \textbf{hyperbolic fold-fold}, \textbf{parabolic fold-fold}, and \textbf{elliptic fold-fold}, respectively.
	
	\begin{figure}[H]
		\centering
		\bigskip
		\begin{overpic}[width=8.5cm]{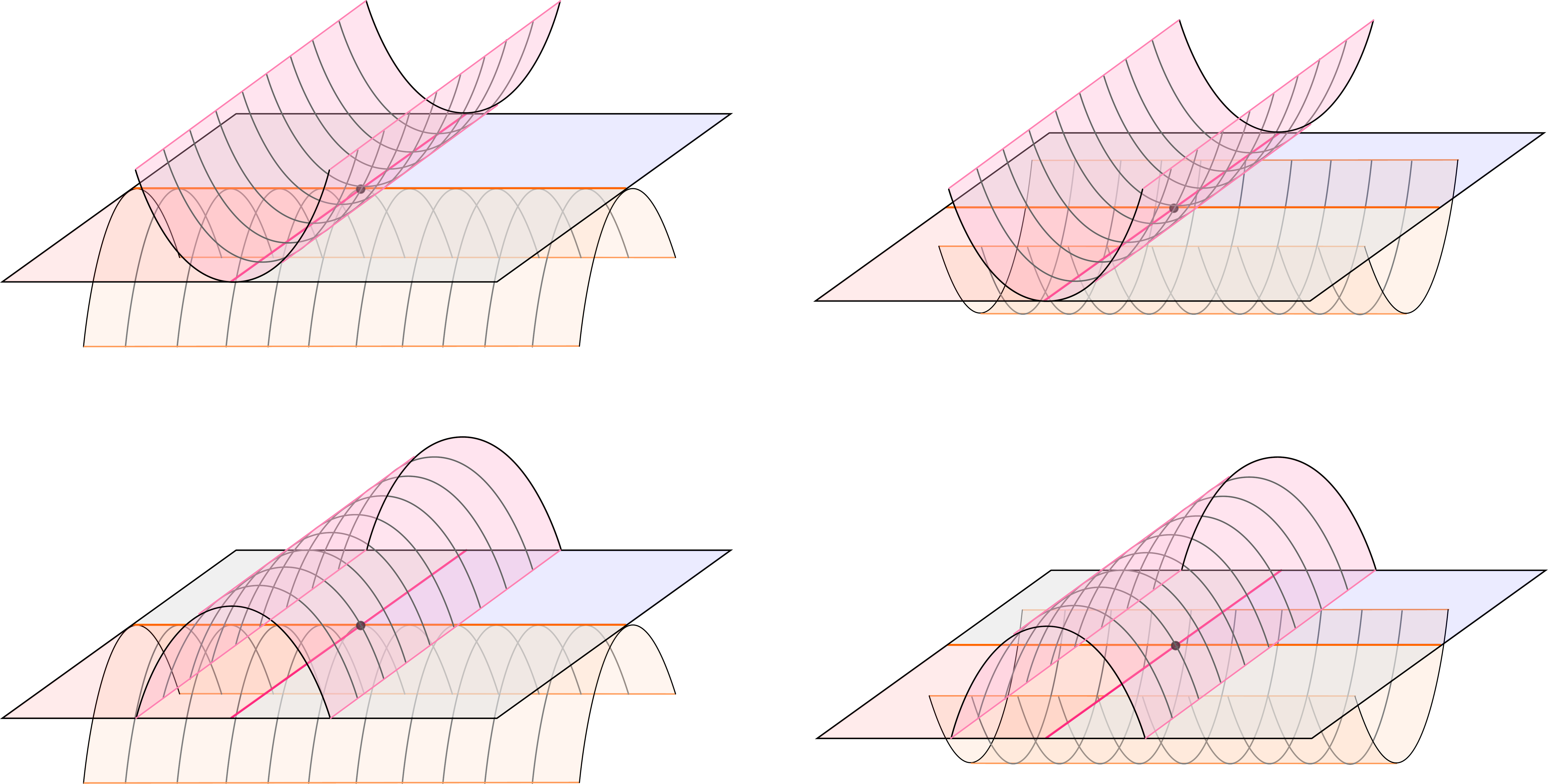}
			\put(22,25){{\footnotesize (a)}}						
			\put(78,25){{\footnotesize (b)}}
			\put(22,-5){{\footnotesize (c)}}						
			\put(78,-5){{\footnotesize (d)}}									
		\end{overpic}
		\bigskip
		\caption{Fold-Fold Singularity: (a) Hyperbolic, (b,c) Parabolic and (d) Elliptic.}	\label{types}
	\end{figure}
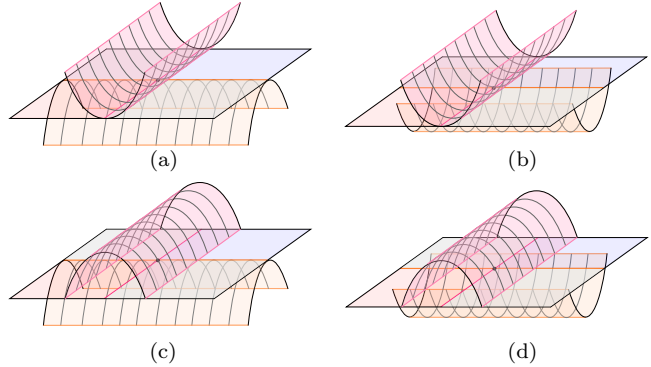 	 			
\end{rem}	

\begin{defn}
	Let $Z_0\in\Or$, we say that $\Gamma_0\subset\s^s$ is a \textbf{$\s$-separatrix of a fold-fold point} $p_0$ of $Z_0$, if it satisfies one of the following conditions:
	\begin{enumerate}
		\item $p_0$ is a singularity of saddle type of $F_{Z_0}^N$ and $\Gamma_0$ is a saddle separatrix of $F_{Z_0}^N$ at $p_0$;
		\item $p_0$ is a singularity of node type of $F_{Z_0}^N$ and $\Gamma_0$ is a strong manifold of $F_{Z_0}^N$ at $p_0$.
	\end{enumerate}	
	If $\Gamma_0$ is both a $\s$-separatrix of two distinct fold-fold points, we say that $\Gamma_0$ is a connection of $\s$-separatrices of fold-fold points.
\end{defn}

\begin{defn}
	Define $\Xi_{0}$ as the set of all $Z\in\Or$ such that, for each $p\in\s$, either $p$ is a regular-regular point of $Z$ or $p$ is an elementary tangential singularity.
\end{defn}

\begin{rem}
	An element $Z\in\Xi_0$ is referred as an \textbf{elementary piecewise-smooth vector field}.
\end{rem}

From Theorem \ref{Vishik}, we derive that:
\begin{prop}\label{vishik_ns}
	$\Xi_{0}$ is an open dense set of $\Or$.
\end{prop}

The elementary tangential singularities of type (FR) and (CR)  determine certain local behavior of the sliding solutions lying on $\s^{s}$, as we can see in the following result proved in \cite{T2}: 

\begin{lem} \label{sliding}
	Let $Z=(X,Y) \in \Omega^{r}$ and assume that $R$ is a connected component of $\Sigma^{s}$. Then:
	\begin{enumerate}
		\item The sliding vector field $F_{Z}$ is of class $\mathcal{C}^{r}$ and it can be smoothly extended beyond the boundary of $R$ through the normalized sliding vector field $F_Z^N$.
		\item If $p\in \partial R$ is a fold point of $X$ and a regular point of $Y$, then $F_{Z}$ is transverse to $\partial R$ at $p$.
		\item If $p\in \partial R$ is a cusp point of $X$ and a regular point of $Y$, then $F_{Z}$ has a quadratic contact with $\partial R$ at $p$.	
	\end{enumerate} 
\end{lem}

\section{Topological Equivalences in $\Or$}\label{equivs}

We are concerned with the persistence of the foliation of the state space generated by a vector field $Z=(X,Y)$ in $\Or$. In light of this, we consider orbital equivalences throughout this work.

\subsection{Sliding Topological Equivalence}

Firstly, we consider a topological equivalence to relate piecewise-smooth vector fields which present the same behavior in the sliding region.

\begin{defn}\label{slequivalencedef}
	Let $Z_{0},Z\in\Or$ be two germs of piecewise-smooth vector fields at $\s$. We say that $Z_{0}$ is \textbf{sliding equivalent} to $Z$ if there exists a homeomorphism $h:\s\rightarrow\s$, which carries $S_{Z_0}$ onto $S_{Z}$ preserving the topological type of singularity and sliding orbits of $Z_{0}$ onto sliding orbits of $Z$. 
\end{defn}

The concept of sliding structural stability is defined in a natural way. We stress that this kind of stability is only concerned with the sliding features (contained in $\s$) of $Z_0\in\Or$.

The set of all sliding structurally stable piecewise-smooth vector fields will be denoted by $\Or_{SLR}$.

\subsection{Semi-Local Topological Equivalence}

In the literature, the local topological equivalence is commonly used to relate piecewise-smooth vector fields which present the same behavior around a point. In this work, we shall consider an extension of this type of equivalence with the purpose of understanding the behavior of a piecewise-smooth vector field around a compact set.

\begin{defn}\label{semiequivalencedef}
	Let $N\neq\emptyset$ be a compact subset of $\s$ and let $Z_{0},Z\in\Or$. We say that $Z_{0}$ is \textbf{semi-locally equivalent} to $Z$ at $N$ if there exist a neighborhood $U$ of $N$ in $M$ and a $\s$-invariant homeomorphism $h:U\rightarrow U$ which carries orbits of $Z_{0}$ onto orbits of $Z$.
\end{defn}

The concept of semi-local structural stability at a compact subset $N$ of $\s$ is defined in a natural way.

We remark that the local term is frequently used to refer phenomena occurring around a point, and for this reason we use the semi-local term to refer a phenomenon occurring in a neighborhood (in $M$) of a compact set. 

In particular, if $N$ is a point of $\s$, say it $p$, then Definition \ref{semiequivalencedef} turns out to be the classical local topological equivalence at a point $p\in\s$, which is extensively studied in \cite{BPS,CJn,CJ1,CJ2,GT,T2,V}. 
It follows from \cite{F,GTS} that each $Z_0\in\Or$ is locally structurally stable at regular-regular, fold-regular and cusp-regular points. In addition, in \cite{GT}, one can find a complete intrinsic characterization of piecewise-smooth vector fields which are locally structurally stables at fold-fold points.

Notice that if $N=\s$, then the semi-local equivalence is quite different from the sliding equivalence. Indeed, the sliding equivalence is concerned only with the elements lying on $\s$ (dimension $2$), whereas the semi-local equivalence at $\s$ regards all the orbits lying in an open set of $M$ (dimension $3$) containing $\s$.

\section{$\s$-Blocks}

The main purpose of this work is to classify all $Z\in\Or$ which are semi-locally structurally stable at $\s$. Now, we introduce a formal language to deal with this problem. We highlight that it is useful to prove the results obtained in the present paper for piecewise-smooth vector fields having a non-simply connected switching manifold (e.g. 2-dim torus). Also, it can be easily adapted to attack the problem in higher dimension. 

The following definition is motivated by the isolating blocks considered in \cite{CE}.

\begin{defn}\label{sigmablock}
	A subset $U\neq\emptyset$ of $\s$ is said to be a \textbf{$\s$-block} of $Z\in\Or$ if $U$ is a compact connected set such that: 
	\begin{enumerate}
		\item $\mu(S_{Z})=0$, where $\mu$ is the Borel measure on $\s$ (with respect to the euclidean metric defined on $\s$);
		\item $\textrm{int}(U)$ is a 2-dimensional manifold;
		\item $\textrm{int}(U)$ is $Z$-invariant;
		\item $\textrm{int}(U)$ is maximal, i.e., every neighborhood of $\textrm{int}(U)$ in $\s$ is not $Z$-invariant.
	\end{enumerate}
	In addition, if $U=\s$, then $U$ is said to be a \textbf{trivial $\s$-block} of $Z$.
\end{defn}

\begin{rem}
	Notice that, if we drop condition $1$ in Definition \ref{sigmablock}, we may have degenerated situations. As an example, we point out the system $X(x,y,z)=(-y,x,0)$, $Y(x,y,z)=(x,y,z)$ and $f(x,y,x)=x^2+y^2+z^2-1$. In this case, $S_{Z}=\s$, and $X$ induces a dynamics on $\s$ which is $Z$-invariant. It is easy to see that it is an structurally unstable situation (consider the pertubation $X_{\e}(x,y,z)=(-y,x,\e z)$). Also, condition $1$ is satisfied for every $Z\in\Xi_0$.
\end{rem}

Notice that, a $\s$-block of $Z\in\Xi_{0}$ is a connected component of $\overline{\s^{s}(Z)}$. Also, if $Z$ has a trivial $\s$-block, then $S_{Z}=\emptyset$. In this case, either $\s=\s^{ss}$ or $\s=\s^{us}$.

\begin{prop}\label{nosigma}
	If $Z_0=(X_0,Y_0)\in\Xi_0$ has no $\s$-blocks, then $S_{Z_0}=\emptyset$, $\s=\s^{c}$ and $Z_0$ is semi-locally structurally stable at $\s$.
\end{prop}	
\begin{proof}
	In fact, if $S_{Z_0}\neq\emptyset$, then from Theorem \ref{Vishik}, it follows that $\s^{s}$ has non-empty interior in $\s$, which means that $Z_0$ would have at least one $\s$-block. It follows that $S_{Z_0}=\emptyset$ and consequently $\s=\s^{c}$.
	
	From the continuity of the maps $F,G:\Xr\times\s\rightarrow \R$, given by $F(X,p)=Xf(p)$ and $G(Y,p)=Yf(p)$, and compactness of $\s$, there exist neighborhoods $\U$ of $X_0$ and $\V$ of $Y_0$, such that $Xh(p)Yh(p)>0$, for each $X\in\U$, $Y\in\V$ and $p\in \s$.
	
	Therefore, $\s^{c}(Z)=\s$, and $Z_0$ and $Z$ are semi-locally equivalents at $\s$,  for each $Z=(X,Y)\in\U\times\V$.  
\end{proof}	

\begin{defn}
	A vector field $Z_0\in\Or$ is said to be \textbf{$\s$-block structurally stable} if either $Z_0$ has no $\s$-blocks or $Z_0$ is semi-locally structurally stable at each $\s$-block of $Z_0$. denote the set of all $Z\in\Or$ which are $\s$-block structurally stable by $\Or_{\s}$.
\end{defn}	

\begin{prop}\label{reduction}
	Let $Z_0\in\Xi_0$. Then, $Z_0$ is $\s$-block structurally stable if and only if $Z_0$ is semi-locally structurally stable at $\s$.
\end{prop}
\begin{proof}
	To prove the non-trivial implication, assume that $Z_0$ is $\s$-block structurally stable. If $Z_0$ has no $\s$-blocks then, from Proposition \ref{nosigma}, $Z_0$ is semi-locally structurally stable at $\s$.
	
	Let $U_1,\cdots, U_k$ be all the $\s$-blocks of $Z_0$. From hypothesis, for each $i=1,\cdots,k$, there exists a neighborhood $\U_i$ of $Z_0$ in $\Or$ such that $Z_0$ and $Z$ are semi-locally equivalent at $U_i$, for each $Z\in\U_i$.
	
	Take $\U= \U_1\cap\cdots\U_k$, and let $Z\in\U$. Hence, there exist disjoint compact neighborhoods $V_i$ of $U_i$ in $M$, and homeomorphisms $h_i:V_i\rightarrow V_i$ which carries orbits of $Z_0$ onto orbits of $Z$, for each $i=1,\cdots,k$. Notice that $Z_0$ and $Z$ are transverse to $\partial V_i\cap\s$ and we may construct $h_i$ such that $h_i|_{\partial V_i}=id$.
	
	Now, let $V=V_1\cup\cdots\cup V_k$. Since $Z_0\in\Xi_0$, $\s\setminus V\subset\s^c(Z_0)$ and $\s\setminus V\subset\s^c(Z)$. Setting $h|{V_i}=h_i$ and $h|_{\s\setminus V}=id$, we can use the flows of $Z_0$ and $Z$ to construct a homeomorphism $h:U\rightarrow U$ (see \cite{GTS}), carrying orbits of $Z_0$ onto orbits of $Z$, where $U$ is a neighborhood of $\s$. Hence $Z_0$ is semi-locally structurally stable at $\s$.
\end{proof}

From Proposition \ref{reduction}, to classify the vetor fields in $\Or$ which are robust around $\s$, it is suficient to understand the $\s$-block structurally stable systems. Notice that all results in this section remain valid if we drop the simply connectedness of the switching manifold.

\section{Main Goal and Statement of the Main Results}

Our strategy is to use informations of $Z\in\Or$ around points of $\s$ to understand its behavior around the switching manifold. In order to do this, we use the concepts of sliding and $\s$-block structurally stability introduced in Section \ref{equivs} to formalize the problem and we give a complete characterization of the sets $\Or_{SLR}$ and $\Or_{\s}$.

Let $\s_0(SLR)$ be the set of $Z=(X,Y)\in\Or$ such that:

\begin{enumerate}
	\item[$G$)] $Z\in\Xi_0$;
	\smallskip	
	\item[$F_{1}$)] If $p\in \s$ is either a hyperbolic or an elliptic fold-fold singularity of $Z$ then $F_{Z}^{N}$ has no center manifold in $V_p\cap\overline{\s^s}$, where $V_p$ is a  neighborhood of $p$ in $\s$;	
	\smallskip	
	\item[$F_{2}$)] If $p\in \s$ is a parabolic fold-fold singularity of $Z$ then $F_{Z}^{N}$ is transient in $V_p\cap\overline{\s^s}$ or it has a hyperbolic singularity at $p$, where $V_p$ is a  neighborhood of $p$ in $\s$;	
	\smallskip												
	\item[$F_{3}$)] There is no connection between $\s$-separatrices of fold-fold points of $F_{Z}$ in $\s^s$;
	\smallskip	
	\item[$F_{4}$)] There is no connection between a $\s$-separatrix of fold-fold and a saddle separatrix of $F_{Z}^{N}$ contained in $\s^{s}$;							
	\smallskip	
	\item[$I_{1}$)] $F_{Z}^N|_{\overline{\s^{s}}}$ has a finite number of pseudo-equilibria. All of them are hyperbolic and they are contained in int($\s^{s}$);
	\smallskip	
	\item[$I_{2}$)] $F_{Z}^N|_{\overline{\s^{s}}}$ has a finite number of periodic orbits. All of them are hyperbolic and they are contained in int($\s^{s}$);
	\smallskip	
	\item[$I_{3}$)] $F_{Z}^N$ doesn't present any saddle connection in $\overline{\s^{s}}$;
	\smallskip	
	\item[$B_{1}$)] There is no orbit of $F_{Z}^N$ contained in $\s^s$ connecting two tangency points of $F_{Z}^N$ with $\partial\s^{s}$; 				
	\smallskip	
	\item[$B_{2}$)] Each saddle separatrix of $F_{Z}^N$ is transversal to $\partial\s^{s}$ (except at fold-fold points).
	\smallskip	
	\item[$R$) ] $F_{Z}^{N}$ has no recurrent orbit.	
\end{enumerate}	

\begin{thmA}[Peixoto's Theorem - Sliding Version]\label{Peixotosl}
	The set $\Or_{SLR}$ is residual in $\Or$ and it coincides with $\s_0(SLR)$.
\end{thmA}

Now, consider the following properties (see \cite{GT} for more details):
\smallskip	
\begin{itemize}
	\item[$\Xi(P)$:] If $p\in\s$ is an invisible-visible fold-fold point of $Z\in\Or$, then the germ of the involution $\phi_{X}$ at $p$ associated to $Z$ satisfies:
	
	\begin{enumerate}
		\item $\phi_{X}(S_{Y})\pitchfork S_{Y}$ at $p$;
		
		\item $F_{Z}^{N}$ and $\phi_{X}^{*}F_{Z}^{N}$ are transversal at each point of $\s^{ss}\cap\phi_{X}(\s^{us})$;
		
		\item $\phi_{X}(S_{Y})\pitchfork F_{Z}^{N}$ at $p$.
	\end{enumerate} 
	\smallskip		
	
	\item[$\Xi(E)$:]  If $p\in\s$ is an elliptic fold-fold point of $Z\in\Or$, then the  germ of the first return map $\phi_{Z}$ at $p$ associated to $Z$ has a fixed point at $p$ of saddle type with both local invariant manifolds $W^{u,s}_{loc}$ contained in $\s^{c}$.
\end{itemize}

\begin{rem}
	If $Z$ has a visible-invisible fold-fold point at $p$, then the roles of $X$ and $Y$ are interchanged in the property $\Xi(P)$.
\end{rem}

Let $\s_{0}$ be the set of $Z\in\Or$ satisfying the following properties:
\begin{itemize}
	\item[$S$)] $Z\in \Or_{SLR}$;
	\item[$F$)] If $p\in\s$ is a fold-fold point of $Z$ then either $\Xi(P)$ or $\Xi(E)$ is satisfied at $p$. 
\end{itemize}

\begin{thmB}[Classification of $\Or_{\s}$] \label{sigmablockclass} The following statements hold:
	\begin{itemize}
		\item[($i$)] $\Or_{\s}=\s_0$;
		\item[($ii$)] $\Or_{\s}$ is not residual in $\Or$;
		\item[($iii$)]$\Or_{\s}$ is residual in $\s(E)$, where $\s(E)$ is the set of $Z\in\Or$ satisfying $\Xi(E)$. Moreover, $\s(E)$ is maximal with respect to this property.		
	\end{itemize}
\end{thmB}

\section{Robustness of Tangency Sets}

In this section we discuss about the structure of the tangency set of $Z\in\Xi_0$. Firstly, we analyze the local behavior of an elementary tangential singularity and secondly, some global features of the tangency set of $Z$ are also discussed.

\subsection{Local Analysis} 

Let $X\in\Xr$ be a $\Cr$ vector field defined around $\s$ (which is the commom boundary of $M^+$ and $M^-$). The local behavior of $X$ at a point $p\in\s$ is a very matured topic and the results of this section can be found in \cite{V,T1,ST} from a different point of view. 

The following propositions provide a geometric interpretation of fold and cusp points in $\s$.

\begin{prop} \label{foldloc}
	Let $X_0\in \Xr$ having a fold point at $p_0\in\s$, then there exists a neighborhood $\V$ of $X_0$ in $\Xr$ and a neighborhood $V$ of $p_0$ in $\s$ such that:
	\begin{itemize}
		\item[(a)] For each $X \in \V$, there exists a unique $\mathcal{C}^{r}$ curve of fold points $\gamma_{X} \subset V$ of $X$ in $\s$ which intersects $\partial V$ transversally at only two points;
		\item[(b)] $p_0 \in \gamma_{X_0}$ and $\sgn(X^2f(p))=\sgn(X_0^{2}f(p_0))$, for each $p\in\gamma_{X}$ and $X\in \V$.  
	\end{itemize}
\end{prop}
\begin{proof}	
	Let $F: \Xr \times \s \rightarrow \R$ given by $F(X,p)= X^{2}f(p)$. It satisfies $F(X_0,p_0)\neq0$. From continuity, there exist neighborhoods $\V_{1}$ of $X_0$ in $\Xr$ and $V_{1}$ of $p_0$ in $\s$ such that $X^{2}f(p)\neq 0$ for each $X \in \V_{1}$ and $p\in V_{1}$, and the sign of $X_0^{2}f(p_0)$ is preserved.	
	
	Let $\phi: (-\e_{1},\e_{1})\times (-\e_{2},\e_{2})\rightarrow V_{2}$ be a local chart of $\s$ around $p_0$ such that $V_{2}\subset V_{1}$, and notice that $\phi$ is a $\Cr$ diffeomorphism. Consider $G:\Xr\times (-\e_{1},\e_{1})\times (-\e_{2},\e_{2})\rightarrow \R$ given by $G(X,x_{1}, x_{2})= Xf(\phi(x_{1},x_{2}))$, and notice that $G$ is a $\mathcal{C}^{r}$ function such that $G(X_0,0,0)=0$.

	Since $X_0^{2}f(p)\neq 0$, we have that $dX_0f(p_0)$ is a nonzero linear transformation, and since $\phi$ is a diffeomorphism, it follows that $\dfrac{\partial G}{\partial (x_1,x_2)}(X_0,0,0)= dX_0f(p_0) \circ d\phi(0,0)$ is nonzero. We conclude that $\dfrac{\partial G}{\partial x_{1} }(X_0,0,0)\neq 0$ or $\dfrac{\partial G}{\partial x_{2}}(X_0,0,0)\neq 0$. 
	
	Without loss of generality, assume that $\dfrac{\partial G}{\partial x_{2}}(X_0,0,0)\neq 0$, now we can use the Implicit Function Theorem (for Banach Spaces) to find a neighborhood $\V$ of $X_0$ in $\Xr$ contained in $\V_1$, real numbers $a,b$ such that $-\e_1<a<0<b<\e_1$, and a $\mathcal{C}^{r}$ function $\alpha:\V\times(a,b)\rightarrow (-\e_{2},\e_{2})$ such that $G(X, x_{1},x_{2})=0$ with $X\in \V$ and $x_{1}\in (a,b)$ if and only if $x_{2}=\alpha(X,x_{1})$.
	
	Notice that, for each $X\in\V$, the curve $c_{X}:(a,b)\rightarrow (a,b)\times(-\e_2,\e_2)$ given by $c_{X}(t)=(t,\alpha(X,t))$, is transverse to each horizontal line $x_{1}=x_0$, with $x_0\in(a,b)$.

	Setting $V=\phi([a_0,b_0]\times(-\e_{2},\e_{2}))$, for some $a<a_0<0<b_0<b$, it follows that, for each $X\in\V$, the curve $\gamma_{X}=\phi\circ c_{X}$ intersects $\partial V$ (transversally) only at the points $\gamma_{X}(a_0)$ and $\gamma_{X}(b_0)$, therefore it satisfies part (a) of the statement. In addition, $\gamma_{X_0}(0)=\phi(0,0)=p_0$, which proves part (b).
	
\end{proof}

\begin{prop} \label{cusploc}
	Let $X_0\in \Xr$ having a cusp point at $p_0\in\s$. Then, there exist a neighborhood $\V$ of $X_0$ in $\Xr$, real numbers $a_0<0<b_0$, and a neighborhood $V$ of $p_0$ in $\s$ such that, for each $X \in \V$:
	\begin{itemize}
		\item[(a)] there exists a unique $\mathcal{C}^{r}$ curve $\gamma_{X}:[a_0,b_0]\rightarrow V$ of tangential singularities of $X$ in $V$ which intersects $\partial V$ transversally at only two points;
		\item[(b)] there exists $a_0<t(X)<b_0$ such that $p(X)=\gamma_X(t(X))$ is a cusp point of $X$. In addition, $\sgn(X^3f(p(X)))=\sgn(X_0^{3}f(p_0))$.		
		\item[(c)] $\gamma_X(t)$ is a fold point of $X$ for every $t\in[a_0,b_0]$ such that $t\neq t(X)$. In addition $X^2f(\gamma_X(t))X^{2}f(\gamma_X(s))<0$, for each $a_0\leq t<t(X)$ and $t(X)<s\leq b_0$.  
	\end{itemize}
\end{prop}
\begin{proof}
	From the linearly independence of $\{df(p_0),dX_0f(p_0),dX_0^2f(p_0)\}$, it follows that $dX_0f(p_0)\neq 0$. Using the same notation and ideas in the the proof of Proposition \ref{foldloc}, we can find neighborhoods $\V$ of $X_0$ in $\Xr$, $V$ of $p_0$ in $\s$, real number $a_0<0<b_0$ and curves $\gamma_X:[a_0,b_0]\rightarrow V$, for each $X\in\V$, such that $$Xf(p)=0, \textrm{ with } X\in\V \textrm{ and } p\in V \Leftrightarrow p=\gamma_X(t), \textrm{ for some } t\in[a_0,b_0],$$
	and $\sgn(X^3f(p))=\sgn(X_0^{3}f(p_0))$, for each $X\in\V$ and $p\in V$.
	
	In addition, $\gamma_X$ intersects $\partial V$ transversally at $\gamma_X(a_0)$ and $\gamma_X(b_0)$, and $\gamma_X(t)\in\textrm{int}(V)$ for each $t\in(a_0,b_0)$. Therefore, item (a) is proved.
	
	To prove (b), consider the $\Cr$ function $H:\Xr\times \rn{3} \rightarrow \rn{3}$ given by $H(X,p)=(f(p),Xf(p),X^2f(p))$. Since $p_0$ is a cusp point of $X_0$, it follows that $H(X_0,p_0)=(0,0,0)$ and $\dfrac{\partial H}{\partial p}(X_0,p_0)$ is invertible. Now, we use the Implicit Function Theorem for Banach spaces to obtain a $\Cr$ function $\beta:\V\rightarrow V$ (reduce the initial neighborhoods $\V$ and $V$, if necessary) such that $H(X,p)=(0,0,0)$, with $X\in\V$ and $p\in V$, if and only if $p=\beta(X)$. 
	
	Reducing $\V$ to $\V\cap \beta^{-1}(\textrm{int}(V))$, we conclude that, each $X\in\V$ has a unique cusp point $p(X)=\beta(X)$ in $V$ which is contained in $\textrm{int}(V)$. Therefore, since $Xf(p(X))=0$, it follows that, there exists $t(X)\in(a_0,b_0)$ such that $\gamma_X(t(X))=p(X)$.
	
	To prove $(c)$, notice that, if $t\neq t(X)$, then $H(X,\gamma_X(t))\neq (0,0,0)$ and $f(\gamma_X(t))=Xf(\gamma_X(t))=0$. Therefore, $X^{2}f(\gamma_X(t))\neq 0$ and $\gamma_X(t)$ is a fold point.
	
	Let $X\in\V$, then $h(t)=X^{2}f(\gamma_X(t))$ is a real smooth function such that $h(t(X))=0$ and $h(t(X))\neq 0$ otherwise. Notice that $h'(t)=dX^{2}f(\gamma_X(t))\cdot \gamma_X'(t)$.
	
	If $h'(t(X))=0$, then $dX^{2}f(p(X))$ is orthogonal to $\gamma_X'(t(X))$, and since $Xf(\gamma_X(t))=f(\gamma_X(t))=0$ for every $t\in[a_0,b_0]$, it follows that $dXf(p(X))$ and $df(p(X))$ are orthogonal to $\gamma_X'(t(X))$. Since $\textrm{span}\{\gamma_X'(t(X))\}^{\bot}$ has dimension $2$, we have that $\{df(p(X)),\ dXf(p(X)),\ dX^{2}f(p(X))\}$ is linearly dependent, which is a contradiction because $p(X)$ is a cusp point of $X$. 
	
	Therefore, $h'(t(X))\neq 0$ and $h(t(X))=0$. If follows that $h(t)h(s)<0$ for each $-\e<t<t(X)<s<\e$, for some $\e>0$ sufficiently small. 
\end{proof}

\subsection{Global Analysis} Now, the tangency set $S_{X}$ of $X\in\Xr$ is analyzed. We shall prove that a $\s$-block of a piecewise-smooth vector field $Z\in\Xi_0$ is robust under small perturbations in $\Or$. 

\begin{prop} \label{SXform}
	If $X\in\Xr$ is simple and $S_{X}\neq \emptyset$, then there exists $n\in\mathbb{N}$ such that $S_{X}=\sqcup_{i=1}^{n} S_{X}^{i}$, where each $S_{X}^{i}$ is diffeomorphic to the unit circle $\mathbb{S}^{1}$. Moreover, $S_{X}$ has at most a finite number of cusp points.
\end{prop}
\begin{proof}	
	From continuity of $Xf$ in $\s$, it follows that $S_{X}= Xf^{-1}(0)$ is a compact subset of $\s$. In addition, from Propositions \ref{foldloc} and \ref{cusploc}, it follows that $S_X$ is locally connected. Therefore, the connected components of $S_{X}$ are open in the induced topology of $S_{X}$ and by compactness, we can conclude that $S_{X}$ has only a finite number of connected components.
	
	Let $\phi: \mathbb{S}^{2}\rightarrow \s$ be a diffeomorphism and consider the $\mathcal{C}^{r}$ function $F: \mathbb{S}^{2}\rightarrow \R$  given by $F(p)=Xf(\phi(p))$.
	
	Notice that, $F^{-1}(0)= \phi^{-1}(S_{X})$, and $dF(x)= dXf(\phi(x))\circ d\phi(x)$ for every $x\in \mathbb{S}^{2}$. If $p\in F^{-1}(0)$, then $\phi(p)\in S_{X}$, and since $S_{X}$ is composed by fold and cusp points, we have that $dXf(\phi(p))\neq 0$ and $d\phi(p)$ is an isomorphism, which implies that $dF(p)\neq 0$. Then $0$ is a regular value of $F$.
	
	Hence, $F^{-1}(0)$ is a 1-dimensional embedded submanifold  of $\mathbb{S}^{2}$. Since $S_{X}$ has a finite number of connected components, if follows that $F^{-1}(0)$ has a finite number of connected components.
	
	We know that every connected component is a closed set in $\s$, which means that each connected component of $F^{-1}(0)$ is a compact connected 1-dimensional embedded submanifold of $\mathbb{S}^{2}$, and thus, diffeomorphic to $\mathbb{S}^{1}$. 
	
	Finally, we can use Propositions \ref{foldloc} and \ref{cusploc} to construct an open covering of $S_{X}$ such that each element of this covering has only fold points with at most one cusp point. By compactness, we conclude that $S_{X}$ has just a finite number of cusp points.
\end{proof}

\begin{figure}[H]
	\centering
	\bigskip
	\begin{overpic}[width=4cm]{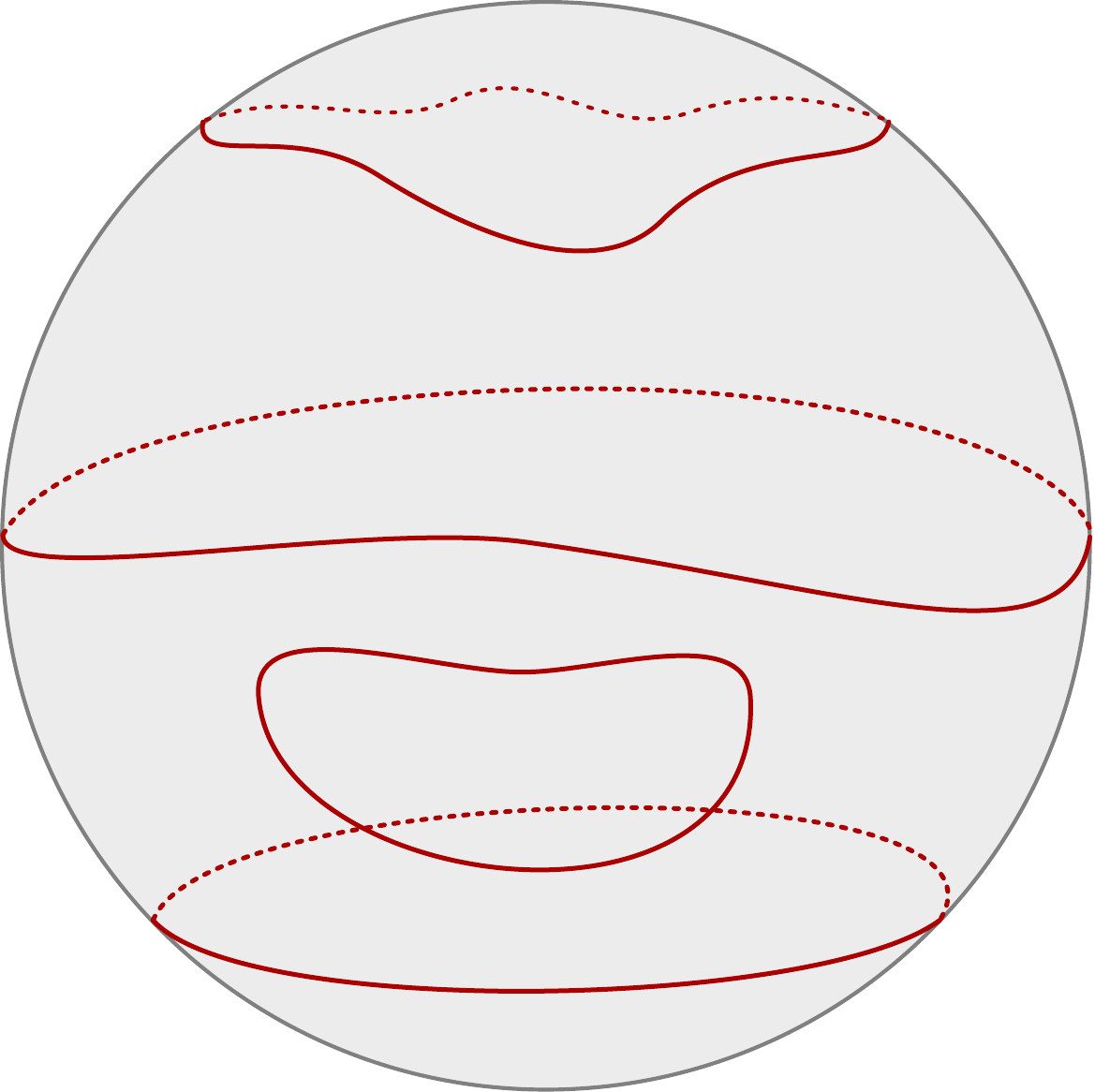}
		\put(2,85){\footnotesize $\s$}
	\end{overpic}
	\bigskip
	\caption{Structure of the tangency set $S_{X}$ of a simple vector field $X\in \Xr$.  }
\end{figure} 

\begin{rem}\label{remcusp}
	Since $\s$ is a compact manifold, the tangency set $S_X$ of $X\in\Xr_S$ with $\s$ is diffeomorphic to a union of circles. From Proposition \ref{cusploc}, the cusp points occur as isolated points in a circle of fold points. 	In addition, if $p$ is a cusp of $X$, then, there exists a smooth curve of fold points of $X$ in $\s$ passing through $p$, which has their visibility changed at $p$. Therefore, the number $k$ of cusp points in a fold circle is always even and it has $k/2$ arcs of visible fold points and $k/2$ arcs of invisible fold points. 
\end{rem}

Now, we prove the persistence of the connected components of the tangency set $S_{X}$ of $X\in\Xr_S$.

\begin{prop} \label{SXcon}
	Let $X_0\in \Xr$ be a simple vector field such that $S_{X_0}\neq\emptyset$, and let $C_0$ be a connected component of $S_{X_0}$ containing $k_0$ cusp points. Then there exist neighborhoods $\V$ of $X_0$ and $V$ of $C_0$ in $\s$  such that,  for each $X\in \V$:
	\begin{itemize}
		\item[(a)] $S_{X}$ has a unique connected component in $V$ containing exactly $k_0$ cusp points.
		\item[(b)] The number of connected components of $S_X$ and $S_{X_0}$ coincide for any $X\in\V$.
	\end{itemize}
	
\end{prop}
\begin{proof}
	Given $p\in C_0$, from Propositions \ref{foldloc} and \ref{cusploc}, there exist neighborhoods $V_{p}\subset \s$ of $p$ and $\V_{p}\subset\Xr$ of $X_0$ such that, for each $X\in \V_{p}$, there exists a smooth curve $\gamma_{X}^{p}: [a_p,b_p]\rightarrow V_{p}$ satisfying the following properties:
	\begin{itemize}
		\item[$(i)$] $\gamma_{X}^{p}(t(X))$ is a point of same nature of $p$, for some $a_p<t(X)<b_p$;
		\item[$(ii)$] $\gamma_X^p(t)$ is a fold point of $X$ for each $t\neq t(X)$.
	\end{itemize}
	
	In addition, this curve contains all the tangency points of $X$ inside $V_{p}$, and intersects $V_p$ transversally at $\gamma_X^p(a_p)$ and $\gamma_X^p(b_p)$. Notice that $\gamma_{X_0}^{p}$ is a local parametrization of $C_0$ at $p$. 
	
	Clearly, $\mathcal{U}=\{V_{p}; p\in C_0\}$ covers $C_0$, and from compactness, we can extract a finite subcovering of $C_0$. Then, $C_0\subset V_{1}\cup\cdots\cup V_{k}$, with $V_i\in\U$, $i=1,\cdots,k$.
	
	From connectedness of $C_0$, and the fact that $C_0$ is a circle, we can order $V_{1},\cdots, V_{k}$ such that, for each $1\leq i \leq k-1$, there exists $p_{i}\in C_0$ such that $p_{i}\in \textrm{int}(V_{i}\cap V_{i+1})$ and $p_{k}\in C\cap V_{k}\cap V_{1}$.
	
	From the construction of the neighborhoods, $p_{i}$ is contained in both curves $\gamma_{X_0}^{i}$ and $\gamma_{X_0}^{i+1}$, and from continuity of $\gamma$ on $X$, we can reduce the neighborhoods $\V_{i}$ and $\V_{i+1}$ such that $\gamma_{X}^{i}$ and $\gamma_{X}^{i+1}$ have at least a  point (for each curve) in $V_{i}\cap V_{i+1}$, for each $X\in\V_i\cap\V_{i+1}$. From uniqueness, $\gamma_{X}^{i}$ and $\gamma_{X}^{i+1}$ must coincide in $V_{i}\cap V_{i+1}$.
	
	Let $V=V_{1}\cup\cdots\cup V_{k}$ and $\V=\V_{1}\cap\cdots\cap \V_{k}$, therefore, for each $X\in\V$ we may construct a $\Cr$ curve $\gamma_{X}:[a_{1},b_{n}]\rightarrow V$ which is injective in $(a_1,b_n)$ such that $\gamma_X(a_1)=\gamma_X(b_n)$ and $\gamma_{X}([a_{1},b_{n}])\cap V_{i}=\textrm{Im}(\gamma_X^{i})$, $i=1,\cdots,k$. 
	
	It follows that, for each $X\in\V$, $C=\textrm{Im}(\gamma_X)$ is a connected component of $S_{X}$ in $V$ and $p\in V$ is a tangency point of $X$ if and only if $p\in C$. The part (a) of the statement is proved.
	
	To prove (b), let $C_{1}^0, \cdots, C_{k}^0$ be all the connected components of $S_{X_0}$. From (a), there exist disjoint open neighborhoods $W_{i}$	of $C_{i}^0$ and $\W_{i}$ of $X_0$ such that, for every $X\in \W_{i}$, $S_{X}$ has a unique connected component contained in $W_{i}$.
	
	Let $W=W_1\cup\cdots\cup W_k$ and $\W= \W_1\cap\cdots\cap \W_k$. Notice that $X_0f(p)\neq 0$ for each $p\in \s\setminus W$. From continuity, for each $p\in \s\setminus W$, there exist neighborhoods $V_p$ of $p$ in $\s$ and $\V_{p}$ of $X_0$ in $\Xr$, such that $Xf(q)\neq 0$, for each $X\in \V_{p}$ and $q\in V_p$.
	
	From compactness of $\s\setminus W$, one can find a neighborhood $\V\subset \W$ of $X_0$ such that $Xf(p)\neq 0$, for each $X\in\V$ and $p\in\s\setminus W$. Therefore, $S_{X}\subset W$, for each $X\in\V$, and we are done.
	
\end{proof}

Notice that Proposition \ref{SXcon} concerns with smooth vector fields defined in manifolds with boundary. Now, we extend this analysis to elementary piecewise-smooth vector fields.

Let $Z\in\Xi_0$, if $C$ is a connected component of $S_{Z}$ composed only by fold-regular and cusp-regular points, then the normalized sliding vector field $F_{Z}^N$ of $Z$ is transversal to $C$, except at cusp points. Indeed, at each cusp-regular point $p$, $F_{Z}^N$ has a quadratic contact with $C$. 

At each quadratic contact of $F_{Z}^N$ with $C$, say it $p$, the orientation of the orbits of $F_Z^N$ reaching $C$ is changed in a neighborhood $V$ of $p$. More specifically, $F_{Z}^N$ reaches $C\cap V$ in either negative or positive time, depending on the side of $C\setminus\{p\}$. Also, if we compute $F_{Z}^N$ along $C$, it gives a complete turn between two cusp-regular points.

Hence, the classical index $I(F_{Z}^N,S_{X})$ of $F_{Z}^N$ along $C$ provides the number of complete turns that it gives along $C$, and from Remark \ref{remcusp}, we conclude that it coincides with half of the number of cusp-regular points of $C$.

Based on this discussion, we have the following result.

\begin{prop}
	Let $Z=(X,Y)\in \Xi_0$. On each connected component $C$ of $S_{Z}$, composed by fold-regular and cusp-regular points, the number of cusp-regular points of $Z$ in $C$ is given by:
	\begin{equation}
	N_{cusp}=2 I(F_{Z}^N, C),
	\end{equation}
	where I is the index of $F_{Z}^N$ along $C$.
\end{prop}

Recalling that the fold-fold points of $Z\in\Xi_0$ are isolated in $\s$, we obtain the next result directly from Proposition \ref{SXform} and compactness of $\s$.


\begin{prop}\label{SZ}
	Let $Z=(X,Y)\in\Xi_0$ such that $S_{Z}\neq \emptyset$. Then:
	\begin{itemize}
		\item[(a)] there exists $n\in\mathbb{N}$ such that $S_{Z}=\cup_{i=1}^{n} S_{Z}^{i}$, where each $S_{Z}^{i}$ is diffeomorphic to the unit circle $\mathbb{S}^{1}$ which is contained either in $S_{X}$ or in $S_{Y}$. In addition, $S_{Z}$ has at most a finite number of cusp-regular and fold-fold points.
		\item[(b)] $p\in S_Z$ is a fold-fold point of $Z$ if and only if $p$ is contained in the intersection of two circles $S_{Z}^{i}$ and $S_{Z}^{j}$, for some $1\leq i,j\leq n$. In this case, $S_Z^i\subset S_X$ and $S_Z^j\subset S_Y$. 
	\end{itemize}	 
\end{prop}

\begin{figure}[H]
	\centering
	\bigskip
	\begin{overpic}[width=7cm]{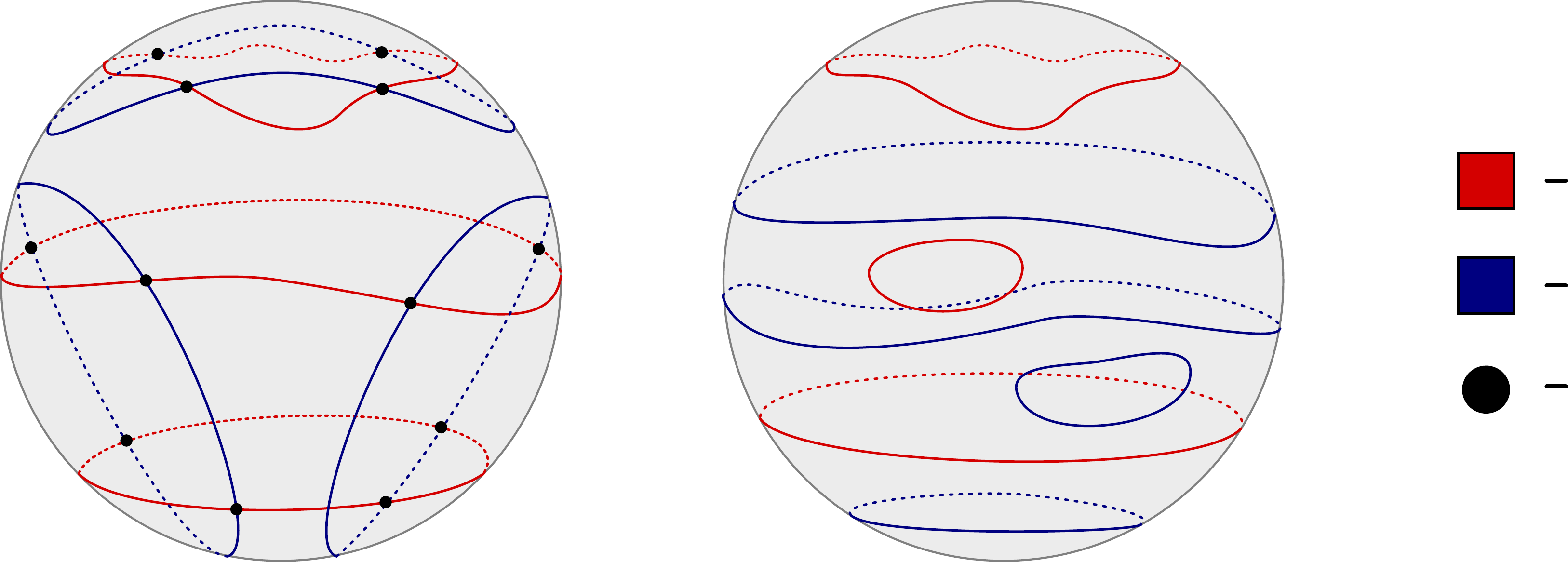}
		\put(102,23.2){{\footnotesize Tangency set $S_{X}$ of $X$}}	
		\put(102,16.5){{\footnotesize Tangency set $S_{Y}$ of $Y$}}
		\put(102,10){{\footnotesize Fold-fold point of $Z$}}
		\put(-4,33){{\small $(i)$}}				
		\put(43,33){{\small $(ii)$}}																																																													
	\end{overpic}
	\vspace{0.2cm}
	\bigskip
	\caption{ Structure of the tangency set $S_{Z}$ of an elementary piecewise-smooth vector field $Z=(X,Y)$ with $(i)$ and without $(ii)$ fold-fold points}
\end{figure}
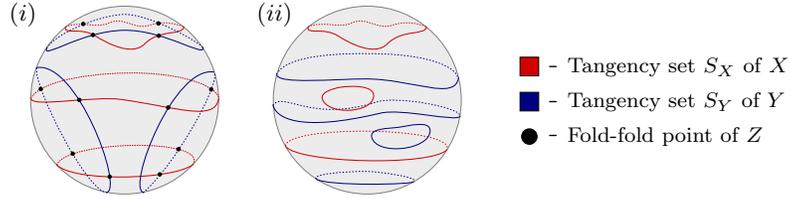 

Combining Proposition \ref{SXcon} with the transversality of $S_X$ and $S_Y$ at fold-fold points of $Z=(X,Y)$, we obtain the following result.

\begin{prop} \label{blocks}
	Let $Z_0=(X_0,Y_0)\in \Xi_0$ be an elementary piecewise-smooth vector field such that $S_{Z_0}\neq\emptyset$, and let $C_0$ be a connected component of $S_{X_0}$ containing $n_0$ cusp-regular points and $m_0$ fold-fold points, then there exist neighborhoods $\V$ of $Z_0$ and $V$ of $C_0$ in $\s$  such that:
	\begin{itemize}
		\item[(a)] for each $Z=(X,Y)\in \V$, $S_{Z}$ has a unique connected component $C$ in $V$ containing exactly
		$n_0$ cusp-regular points and $m_0$ fold-fold points.
		\item[(b)] For each circle $S_0$ of $S_{X_0}$ (resp. $S_{Y_0}$) contained in $C_0$, there exists a neighborhood $V_0\subset V$ of $S_0$ such that each $Z=(X,Y)\in\V$ has a unique circle $S$ of $S_X$ (resp. $S_{Y}$) contained in $V_0$, with the same number of cusp-regular and fold-fold points of $S_0$.  
		\item[(c)] If two circles $S_0^1$ and $S_0^2$ of $S_{X_0}$ contained in $C_0$ intersect themselves at fold-fold points $p_1, p_2,\cdots, p_{2k}$, then, for each $Z\in\V$, the correspondent circles of item $(b)$ intersect themselves at fold-fold points $p_1(Z),\cdots, p_{2k}(Z)$. In addition, $p_i(Z)$ is sufficiently close to $p_i$ and they have the same visibility.
		\item[(d)] for each $Z=(X,Y)\in \V$, $S_{Z}$ has the same number of connected components of $S_{Z_0}$, for each $Z\in\ V$.
	\end{itemize}	
\end{prop}

With this, we see that for a small neighborhood $\V$ of $Z_0\in\Xi_0$, each $Z\in\V$ has a tangency set $S_Z$ with exactly the same characteristics of $S_{Z_0}$, i.e., each circle of $S_Z$ is near from a circle of $S_{Z_0}$ and they present the same configuration of intersections. It allows us to conclude that, if $Z_0$ has a $\s$-block $U_0$, then there exists a neighborhood $V_0$ of $U_0$ in $\s$, such that each $Z\in\V$ has a unique $\s$-block $U$ contained in $V_0$ and it has the same structure of $U_0$.

We complete the characterization of the tangency sets by exhibiting another property concerning the number of fold-fold points of $Z\in\Xi_0$. 

\begin{prop}
	If $Z\in\Xi_0$, then the number of fold-fold points of $Z$ is even.
\end{prop}
\begin{proof}
	In fact, if $p$ is a fold-fold points of $Z$, then $p$ is contained in the transversal intersection of two circles of $S_{Z}$, say it $C_{1}$ and $C_{2}$.
	
	From Jordan Curve Theorem, $C_{1}$ divides $\s-C_{1}$ into two connected components, and since $C_{2}$ is a closed curve, it follows that $\gamma$ has to intersect $C_{1}$ again in another point different from $p$, say it $q$. Since $Z\in\Xi_0$, $q$ is a fold-fold point.
	
	To complete the proof, it is sufficient to notice that if $\tilde{p}$ is another fold-fold point different from $p$ and $q$, then by the same argument as above we can find another fold-fold point $\tilde{q}$ different from the others, and the result follows by induction and the fact that $Z$ has a finite number of fold-fold points (due to compactness of $\s$).
\end{proof}

\begin{rem}
	Due to the similarity on the behavior between a fold-fold point and the vertex defined in \cite{PP}, it will be referred as a \textbf{vertex} of $S_{Z}$. In addition, if $S_{Z}$ has no vertices, then each $\s$-block of $Z\in\Xi_0$ has a smooth boundary.
\end{rem}

Finally, we notice that, if $Z_0\in\Xi_0$, the transversality of $S_{X_0}$ and $S_{Y_0}$ at vertexes can not be dropped. Indeed, if $S_{X_0}$ is tangent to $S_{Y_0}$ at $p$ then  $T_{p}S_{X_0}=T_{p}S_{Y_0}$ and the intersection between $S_X$ and $S_Y$ can be easily broken by translations. In this case, the structure of the tangency set is quite different for small perturbations of $Z_0$.


\section{Sliding Structural Stability}

In this section we discuss the concept of sliding structural stability and we prove Theorem A. 

Let $Z_0\in\Xi_0$ having a $\s$-block $U_0$, as we can see in \cite{GT,GTS}, the first element to construct a semi-local equivalence at $U_0$ between $Z_0$ and some $Z\in\Or$ is the existence of an sliding equivalence between $Z_0$ and $Z$. 
Based on this, we propose a sliding version of Peixoto's Theorem for typical piecewise-smooth vector fields. It is worth to mention that in \cite{PP}, the authors have mentioned that the case where the boundary is piecewise smooth can be considered with the addition of some conditions on the boundary.

The most relevant difference between the classical case and the one to be considered in Theorem A is that the sliding vector field $F_{Z_0}^{N}$ is allowed to have a singularity at vertices (fold-fold points of $Z_0$) of the boundary, since it is persistent at vertices in the discontinuous world. Actually, the boundary changes with the sliding vector field in such a way that the singularity always stays in the vertex.

Recall that $F_{Z_0}$ is defined on $\textrm{int}(U_0)$, but it is not defined on $\partial U_0$. Since $U_0$ is connected, we can use Lemma \ref{sliding} to extend $F_{Z_0}$ to $\partial U_0$ through the normalized sliding vector field $F_{Z_0}^{N}$, and then we can determine the stability of $F_{Z_0}$ on $U_0$ by using the stability of $F_{Z_0}^N$ on $\overline{U_0}$.

Notice that a vertex is a singularity of $F_{Z_0}^N$, but it is not a critical point of $F_{Z_0}$ (the vector field is not even defined on these points), then if a trajectory of $F_{Z_0}$ reaches a vertex, it does in a finite time, differently from a trajectory of $F_{Z_0}^N$.

\subsection{Proof of Theorem A}
Let $Z_0=(X_0,Y_0)\in\s_0(SLR)$.

If $\s=\s^{s}(Z_0)$, then $F_{Z_0}$ is defined in the whole $\s$ and for a small neighborhood $\V$ of $Z_0$ in $\Or$, the sliding vector field $F_{Z}$ of $Z\in\V$ is also defined in $\s=\s^{s}(Z)$. Therefore, the result follows from Lemma \ref{sliding} and the classical version of Peixoto's Theorem.

Assume for instance that $\s^{s}(Z_0)\neq \s$. Since $Z_0$ has a typical tangency set, it follows that $\overline{\s^{s}(Z_0)}$ is a compact set which is properly contained in $\s\setminus\{p_0\}$, for some $p_0\in\s$. Therefore, we can perform a change of coordinates in $\s\setminus\{p_0\}$ (stereographic projection) which brings $\overline{\s^{s}(Z_0)}$ into a compact subset of $\rn{2}$. Denote this identification by $\overline{\s^{s}(Z_0)}\simeq_{\rho} M_{0}$.

\begin{figure}[H]
	\centering
	\bigskip
	\begin{overpic}[width=9cm]{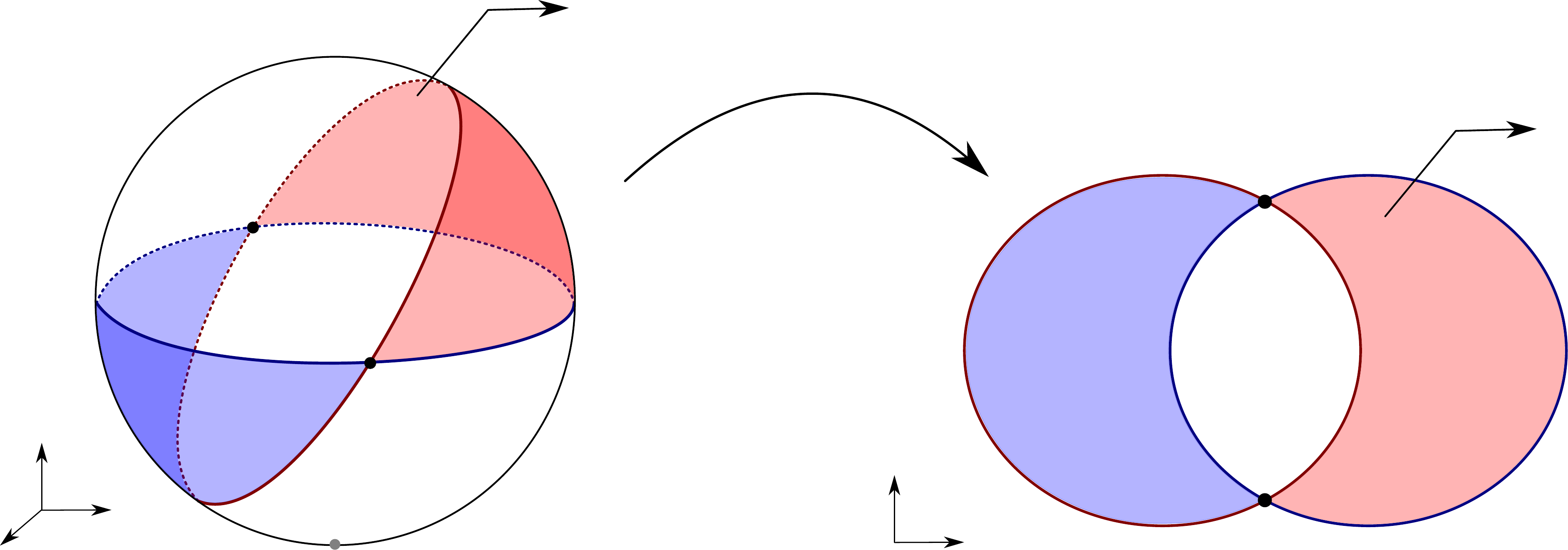}
		\put(5,25){\footnotesize $\s$}
		\put(38,33){\footnotesize $\overline{\s^{s}(Z_0)}$}
		\put(99,26){\footnotesize $M_0$}			
		\put(50,30){\footnotesize $\rho$}				
		\put(20,-2){\footnotesize $p_0$}							
	\end{overpic}
	\bigskip
	\caption{ Interpretation of the identification $\overline{\s^{s}(Z_0)}\simeq_{\rho} M_{0}$. }	
\end{figure}
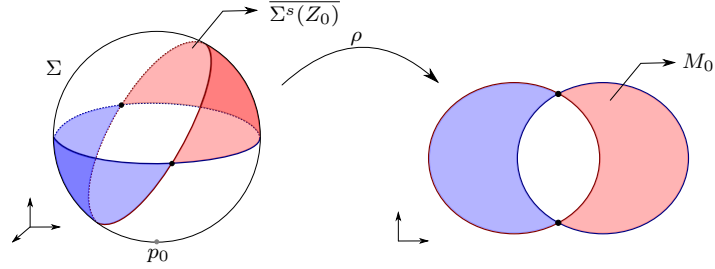  	

Since $Z_0\in\Xi_0$, the boundary $\partial M_0$ is composed by fold-regular points, a finite (even) number $n_c$ of cusp-regular points, $p_i(Z_0)$, $i=1,\cdots,n_c$, and a finite (even) number $n_f$ of fold-fold points, $q_j(Z_0)$, $j=1,\cdots,n_f$.

Notice that each $\s$-block $U_0\subset M_0$ of $Z_0$ is a path-connected region, which has its boundary $\partial U_0$ composed by a union of circles which are pairwise transversal, and they may intersect themselves only at a finite number of points. In addition, each circle is composed by fold-regular, cusp-regular and fold-fold points of $Z_0$, and two circles intersect themselves at $p$ if, and only if, $p$ is a fold-fold point of $Z_0$.

\begin{figure}[H]
	\centering
	\bigskip
	\begin{overpic}[width=5cm]{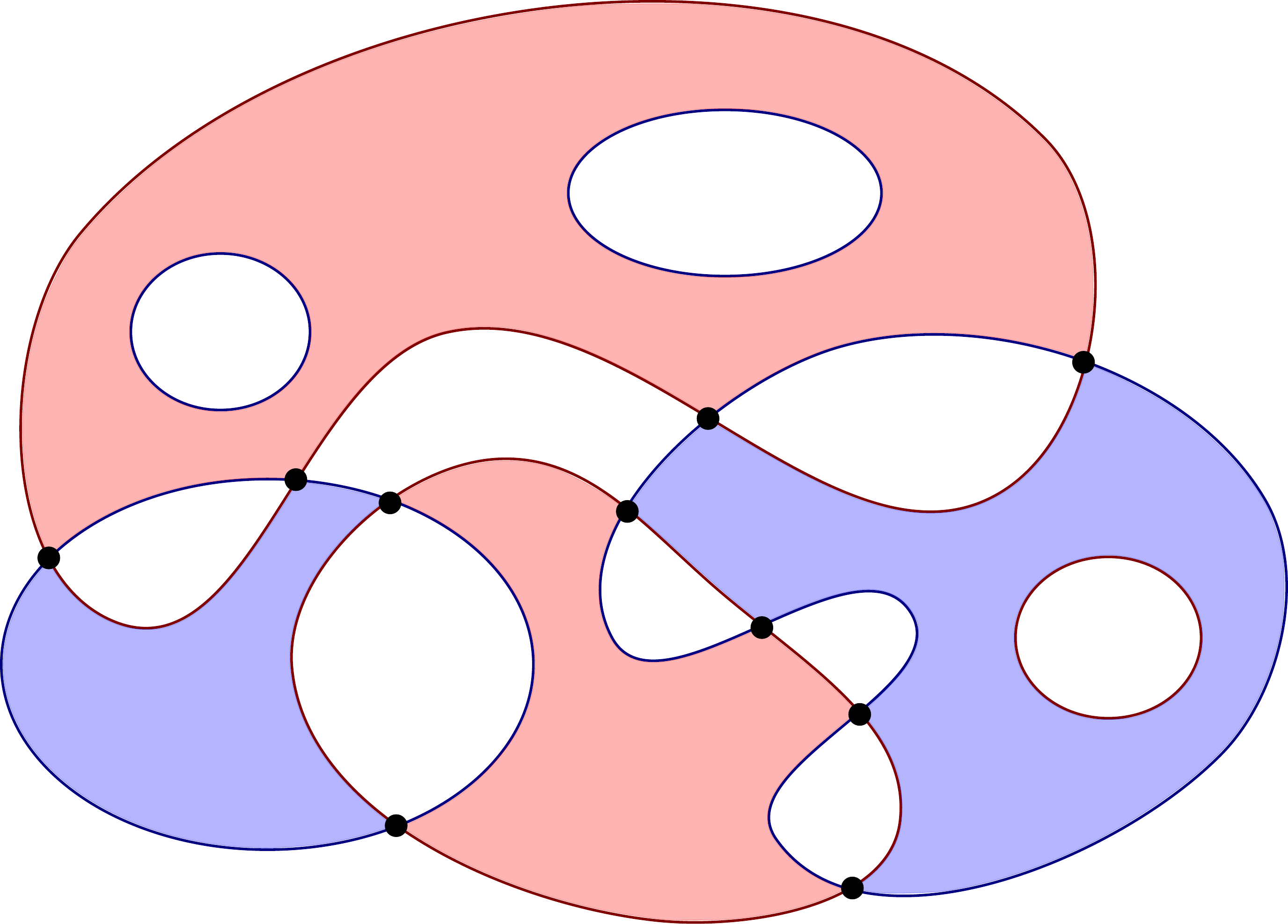}
		
	\end{overpic}
	\bigskip
	\caption{ Example of a $\s$-block $M_0$ of $Z_0=(X_0,Y_0)$. The stable and unstable sliding regions are represented by the colors blue and red, respectively. }	
\end{figure}  	

Now each $\s$-block $U_{0}$ of $Z_0$ is path-connected, but $\widetilde{U_{0}}=U_{0}\setminus\{q_1(Z_0),\ \cdots,\ q_{n_f}(Z_0)\}$ is not connected. Call the closure of the connected components of $\widetilde{U_{0}}$ of all $\s$-blocks $U_0$ of $Z_0$ by $R_i(Z_0)$, $i=1,\cdots,l$, where $l$ is the number of connected components of $\widetilde{U_{0}}$.	

From the characterization of $U_0$, it follows that, each $R_{i}(Z_0)$ is a simple polygon with $a_i\leq n_f$ vertices and a finite number of holes in its interior. Notice that, if $a_i=0$, then $\partial R_{i}(Z_0)$ is smooth. 	

\begin{figure}[H]
	\centering
	\bigskip
	\begin{overpic}[width=11cm]{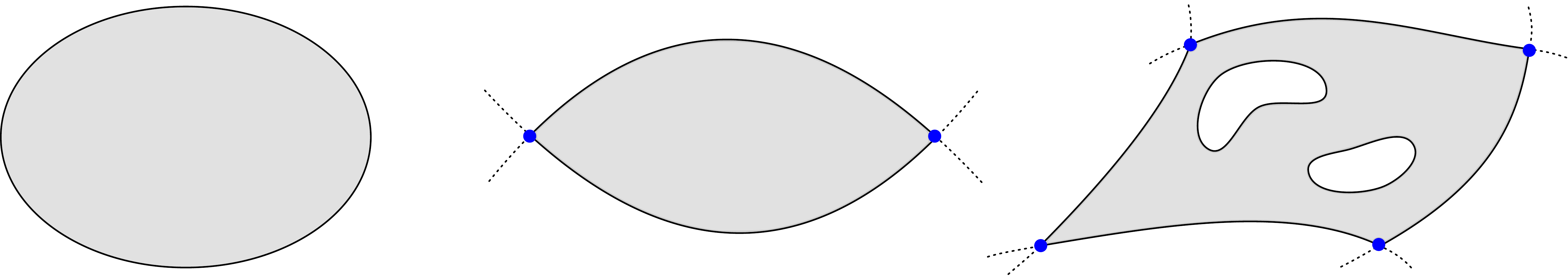}
		\put(11,-3){\footnotesize $(i)$}
		\put(45,-3){\footnotesize $(ii)$}			
		\put(80,-3){\footnotesize $(iii)$}						
	\end{overpic}
	\bigskip
	\caption{ Example of regions $R_i(Z_0)$: $(i)\ a_i=0$, $(ii)\ a_i=2$ and  $(iii)\ a_i=4$.}	
\end{figure}  

From Proposition \ref{blocks}, it follows that, for a small neighborhood $\V$ of $Z_0$ in $\Or$, the region $M_Z=\overline{\s^{s}(Z)}\subset \rn{2}$ (use the same previous change of coordinates) has exactly the same configuration of $M_0$, for each $Z\in\V$. Indeed, $M_Z=R_1(Z)\cup\cdots\cup R_l(Z)$, where $R_i(Z)$ is a region with the same characteristics of $R_i(Z_0)$, i.e., $R_i(Z)$ is a simple polygon with $a_i$ vertices and the same number of holes that $R_i(Z_0)$ in its interior.

Also, for each fold-fold (vertex) $q(Z_0)$ of $R_i(Z_0)$, there exists a unique fold-fold of $q(Z)$ of $R_i(Z)$ with the same visibility of $q(Z_0)$, which is sufficiently close of $q(Z_0)$.

\begin{lem} There exists a homeomorphism $h_{Z}^{i}: R_{i}(Z_0)\rightarrow R_{i}(Z)$ which preserves the type of singularity of the boundary  and carries sliding orbits of $Z_0$ onto sliding orbits of $Z$, for $i=1,\cdots,l$. In addition, $h_Z^i(q_j(Z_0))=q_j(Z)$ for each $q_j(Z_0)\in R_{i}(Z_0)$, $j=1,\cdots,n_f$, $i=1,\cdots,l$.
\end{lem}	
\begin{proof}{
		Let $R$ be one of the regions $R_{i}(Z_0)$ and let $Z\in\V$. 
		
		If $R$ has no vertices, then $\partial R$ is smooth and so there exists a diffeomorphism $\Psi:R\rightarrow\widetilde{R}$, where $\widetilde{R}=R_i(Z)$. Therefore, we may construct a homeomorphism $\widetilde{h}:R\rightarrow R$ between $Z_0$ and $\Psi_*Z$, via the classical Peixoto's Theorem. Hence, $h=\Psi\circ \widetilde{h}$ satisfies the properties of the Lemma.
		
		Now, for simplicity, assume that $R$ is a simple polygon with only two vertices $q_1$ and $q_2$, which has no holes. We stress that there is no loss of generality in this assumption, since the following construction can be easily extended to any configuration of $R$.
		
		Therefore, $\widetilde{R}$ has two vertices $\widetilde{q_1}$ and $\widetilde{q_2}$ which are sufficiently near from $q_1$ and $q_2$, respectively.
		
		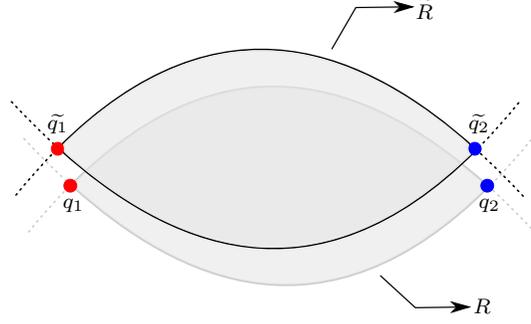
\begin{figure}[H]
			\centering
			\bigskip
			\begin{overpic}[width=7cm]{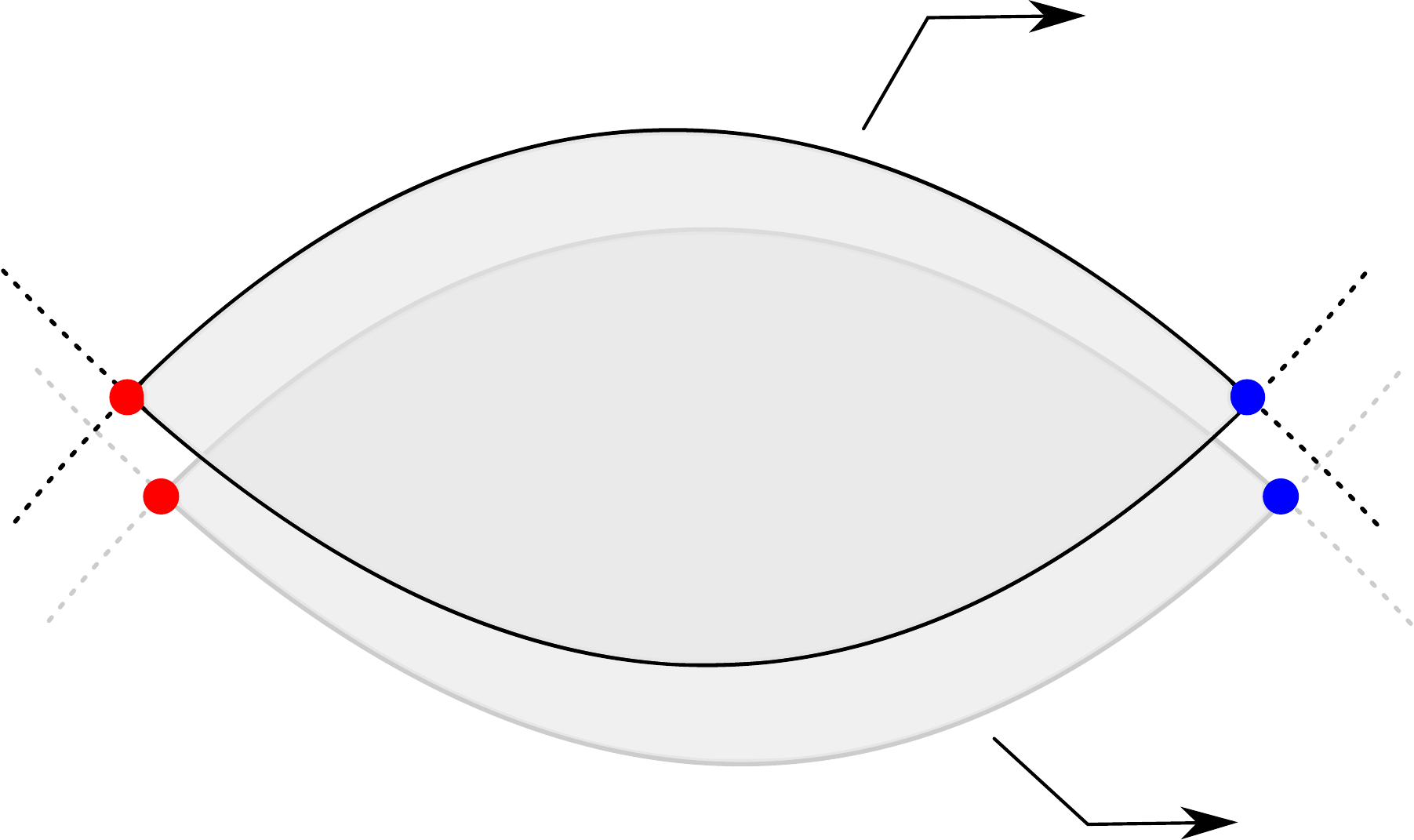}
				\put(10,20){\footnotesize $q_1$}
				\put(7,35){\footnotesize $\widetilde{q_1}$}			
				\put(89,20){\footnotesize $q_2$}
				\put(87,35){\footnotesize $\widetilde{q_2}$}	
				\put(88,0){\footnotesize $R$}
				\put(77,56){\footnotesize $\widetilde{R}$}						
			\end{overpic}
			\bigskip
			\caption{ Persistence of the region $R$.}	
		\end{figure}  
		
		Notice that $\partial R= A_1\cup A_2$, where $A_i$, $i=1,2$, is an open arc with extrema $q_1$ and $q_2$ just composed by fold-regular and cusp-regular points of $Z_0$. Recall that, from Lemma \ref{sliding}, $F_{Z_0}^N$ is transverse to $A_i$ at fold-regular points and it has a quadratic contact with $A_i$ at cusp-regular points, $i=1,2$. Hence, $F_{Z_0}^N$ satisfies all the hypotheses of Peixoto's Theorem, with exception of the points $q_1$ and $q_2$.

		Hence, we must clarify what happens around points $q_1$ and $q_2$, i.e., fold-fold points of $Z_0$. Let $q$ be either $q_1$ or $q_2$, from \cite{GT} (see sliding dynamics' section), we obtain that under hypotheses \textbf{$F_1$} and \textbf{$F_2$}, there exists a neighborhood $V$ of $q$ in $\rn{2}$, such that $A_i$ is transversal to $\partial V$ and $F_{Z_0}^N$ satisfies either one of the following:
		
		\begin{enumerate}
			\item[I $-$] $(i)$ $F_{Z_0}^N$ is transversal to $\partial V\cap R$, $(ii)$ there exists a unique orbit $\Gamma_0$ of $F_{Z_0}^N$ departing from $q$ and reaching $\partial V\cap R$, and $(iii)$ each orbit passing through another point of $V\cap R$ departs from an arc $A_i$, $i=1,2$, and reaches $\partial V\cap R$ at finite time.\vspace{0.2cm}
			\item[II $-$] $(i)$ $F_{Z_0}^N$ is transversal to $\partial V\cap R$, $(ii)$ each orbit passing through a point of $V\cap R$ departs from $(\partial V\cap R) \cup (A_1\cap V) \cup (A_2\cap V)$ and reaches $q$ in infinite positive time.\vspace{0.2cm}
			\item[III $-$] $(i)$ $F_{Z_0}^N$ is transversal to $\partial V\cap R$, with exception of a point $x_0\in\partial V\cap \textrm{int}(R)$, where $F_{Z_0}^N$ has a quadratic contact with $\partial V\cap R$. $(ii)$ Each orbit of $V\cap R$ either departs from $A_1$ and reaches $A_2$ or it departs from $\partial V$ and reaches $\partial V$.\vspace{0.2cm}
			\item[IV $-$] $(i)$ $F_{Z_0}^N$ is transversal to $\partial V\cap R$, $(ii)$ there exists a $\s$-separatrix (type node) $\Gamma_0$ of $q$ which reaches $\partial V\cap R$ at $x_0$, and $(ii)$ the orbit passing through another point of $V\cap R\setminus \Gamma_0$ either departs from $A_{i_1}$ and reaches $\partial V$ or it departs from $q$ and reaches $A_{i_2}\cap \partial V$, $\{i_1,i_2\}=\{1,2\}$.\vspace{0.2cm}
			\item[V $-$] $(i)$ $F_{Z_0}^N$ is transversal to $\partial V\cap R$, with exception of a point $x_0\in\partial V\cap \textrm{int}(R)$, where $F_{Z_0}^N$ has a quadratic contact with $\partial V\cap R$. $(ii)$ There exists two separatrices (type saddle) of $q$ which reaches $\partial V$ at $x_1$ and $x_2$, respectively, and $x_0$ is between $x_1$ and $x_2$. $(iii)$ Each orbit through a point of $V\cap R$ either departs from $\partial V$ and reaches $A_{i_1}$ or it departs from $\partial V$ and reaches $\partial V$ or it departs from $A_{i_{2}}$ and reaches $\partial V$, $\{i_1,i_2\}=\{1,2\}$. 
		\end{enumerate}
		
		\begin{figure}[H]
			\centering
			\bigskip
			\begin{overpic}[width=12cm]{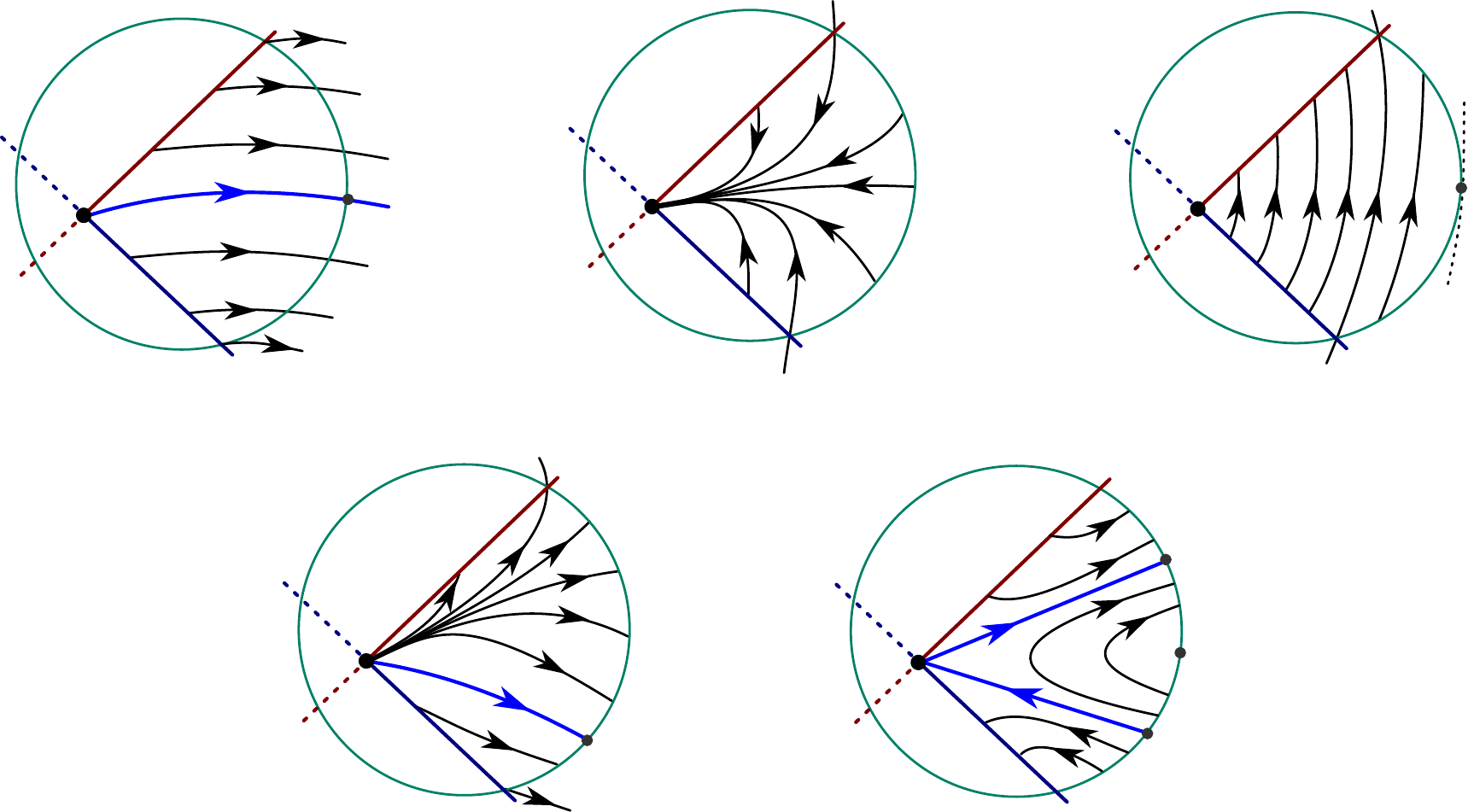}
				\put(27,40){\footnotesize $\Gamma_0$}
				\put(3,52.5){\footnotesize $\partial V$}			
				\put(11,27){\footnotesize (I)}
				\put(49.5,27){\footnotesize (II)}			
				\put(86,27){\footnotesize (III)}	
				\put(101,42){\footnotesize $x_0$}	
				\put(7,35){\footnotesize $A_1$}		
				\put(8.5,47){\footnotesize $A_2$}		
				\put(30,-3){\footnotesize (IV)}	
				\put(41,4){\footnotesize $\Gamma_0$}					
				\put(67.5,-3){\footnotesize (V)}	
				\put(79,4.5){\footnotesize $x_1$}				
				\put(81.5,10){\footnotesize $x_0$}	
				\put(81,17){\footnotesize $x_2$}																						
			\end{overpic}
			\bigskip
			\caption{ Description of the sliding dynamics near a fold-fold point: I-V.}	
		\end{figure}
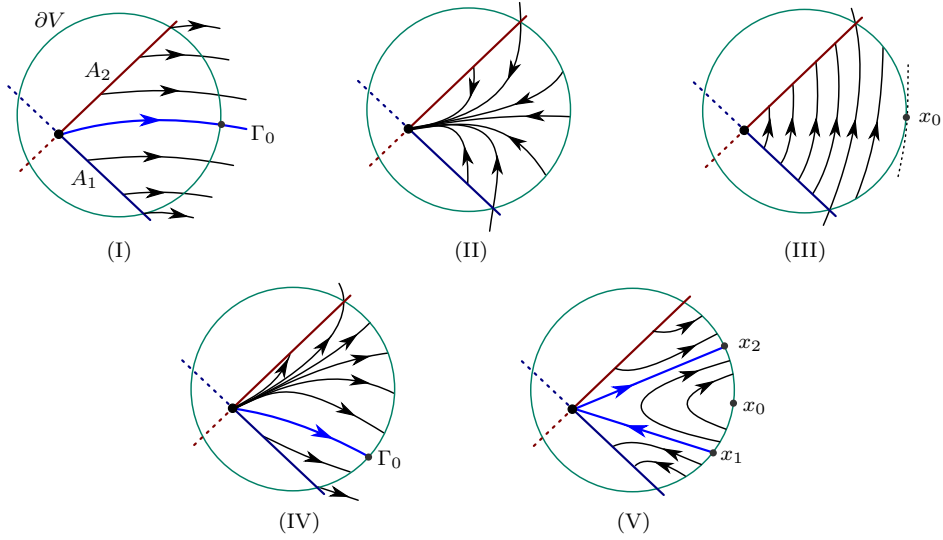  
		
		We remark that all the situations $I-V$ may happen also for negative time.
		
		Let $\widetilde{q}$  the point $\widetilde{q_i}$ which is near from $q$. Since $Z$ is suffciently near from $Z_0$, it follows that $\widetilde{q}\in V$ and $F_Z^N$ also satisfies the same property of $F_{Z_0}^N$ ($I-V$) in $V$. Hence it is straightforward to construct a homeomorphism $h_{q}: \overline{V}\cap R\rightarrow \overline{V}\cap \widetilde{R}$, such that $h_q(q)=\widetilde{q}$, which carries sliding orbits of $Z_0$ onto sliding orbits of $Z$, see \cite{CJ1,GT,T1,T3}.
		
		Notice that $v\in \partial V\cap A_{i}$ is a vertex point in the sense of \cite{PP} (it is not a fold-fold point of $Z_0$), and it satisfies that $F_{Z_0}^{N}$ is transversal to both $\partial V$ and $A_i$ at $v$. Therefore $F_{Z_0}^{N}$ satisfies condition \textbf{B5} of \cite{PP}.
		
		In addition, we can reduce the neighborhoods $V$ of each fold-fold point of $R$ in order that there is no trajectory of $F_{Z_0}^N$ connecting two vertices. Hence it also satisfies condition \textbf{B6} of \cite{PP}. 
		
		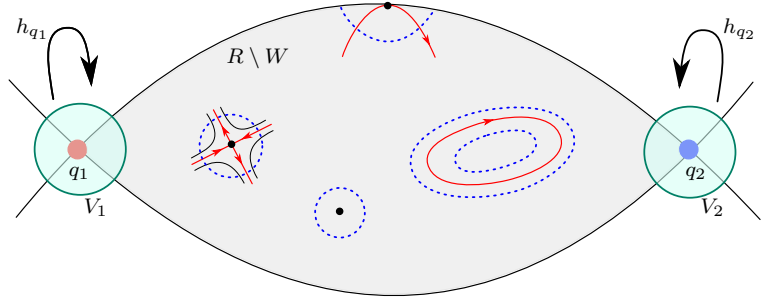
\begin{figure}[H]
			\centering
			\bigskip
			\begin{overpic}[width=10cm]{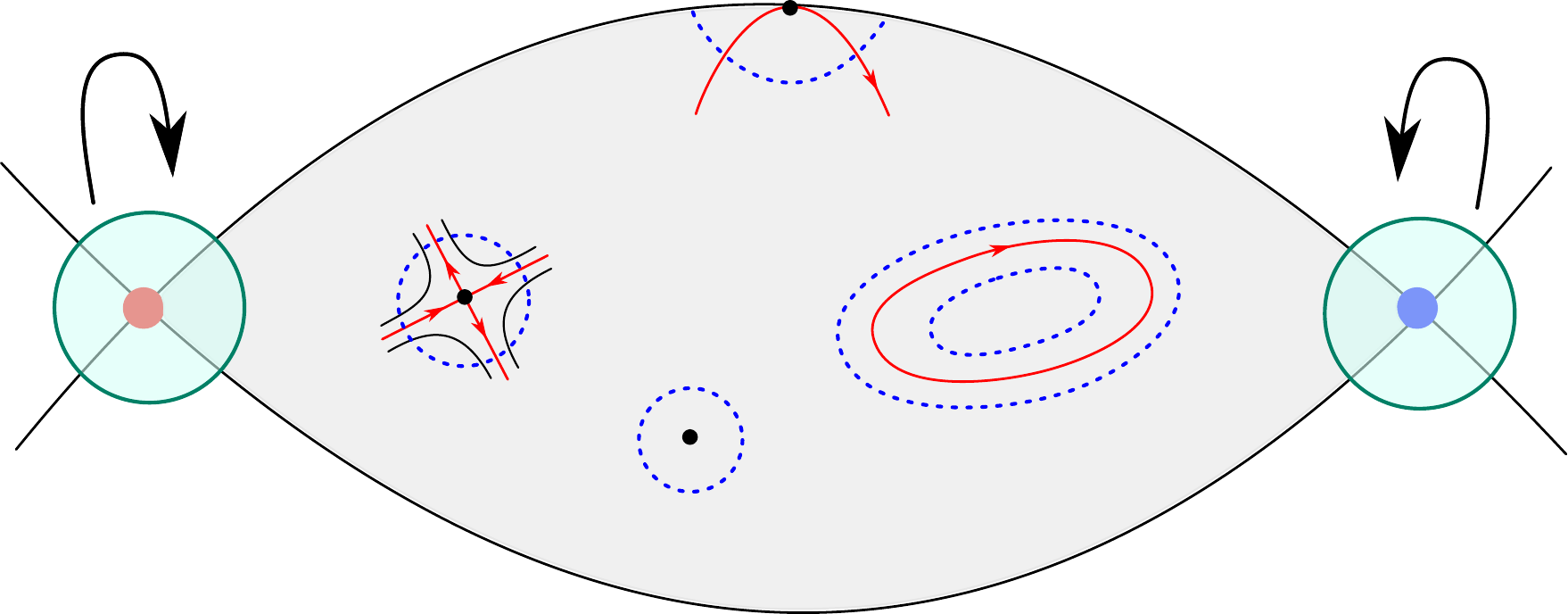}
				\put(1,35){\footnotesize $h_{q_1}$}
				\put(8,16){\footnotesize $q_1$}			
				\put(10,11){\footnotesize $V_1$}		
				\put(95,35){\footnotesize $h_{q_2}$}
				\put(90,16){\footnotesize $q_2$}			
				\put(92,11){\footnotesize $V_2$}			
				\put(29,31){\footnotesize $R\setminus W$}									
			\end{overpic}
			\bigskip
			\caption{ Construction of the homeomorphism $h^i_Z$ using distinguished neighborhoods of hyperpolic equilibrium points, hyperbolic periodic orbits, and quadratic tangecy points  of $F_{Z_0}^N$ in $R\setminus W$ obtained via Peixoto's Theorem.}	
		\end{figure} 
		
		After doing this process for the points $q_1$ and $q_2$, we obtain two neighborhoods $V_1$ and $V_2$ as above, let $W=V_1\cup V_2$. Since $Z_0$ satisfies the hypotheses from Peixoto's Theorem inside $R\setminus W$ we can use the same methods as used in \cite{PP} to extend the homeomorphisms $h_{q_1}$ and $h_{q_2}$ into a homeomorphism $h_{Z}: R\rightarrow \widetilde{R}$ satisfying the properties of the claim.}		
\end{proof}
\bigskip

Therefore, all the homeomorphisms $h_{Z}^{i}$ can be used to construct a homeomorphism $h_Z^{s}:M_0\rightarrow M_Z$. Now, we can return to the initial coordinates of $\s-\{p_0\}$ to obtain a homeomorphism $\widetilde{h_Z^{s}}:\overline{\s^{s}(Z_0)}\rightarrow \overline{\s^{s}(Z)}$. Hence, considering any extension $h_Z:\s\rightarrow\s$ of $\widetilde{h_Z^{s}}$, we conclude that $Z_0\in\Or_{SLR}$.

Therefore, $\s_0(SLR)\subset\Or_{SLR}$. Now, if $Z_0$ doesn't satisfy some condition of $\s_0(SLR)$, then we conclude that $Z_0\notin \Or_{SLR}$ (see \cite{GT,PM,PP}). Therefore, $\Or_{SLR}=\s_0(SLR)$.

Since $\s_0(SLR)$ is a residual set in $\Or$, the proof is complete.

\begin{rem}\label{remsl}
	From the proof of Theorem A, it follows that the construction of the equivalence $h:\s\rightarrow\s$ around fold-fold points can be made in several ways. In particular, any local equivalence $h_q$ at a fold-fold $q$ may extend itself in a sliding equivalence $h:\s\rightarrow\s$ of Definition \ref{slequivalencedef}.
\end{rem}

\section{Proof of Theorem B}

Let $Z_0\in\s_0$, we shall prove that $Z_0$ is $\s$-block structurally stable. If $Z_0$ has no $\s$-blocks, then $Z_0\in\Or_\s$ (see Proposition \ref{nosigma}). Let $U_0$ be a $\s$-block of $Z_0$.

If $U_0=\s$, then condition \textbf{S} of $\s_0$ allows us to find a neighborhood $\V$ of $Z_0$ such that, for each $Z\in\V$, there exists a homeomorphism $h_s:\s\rightarrow\s$ carrying sliding orbits of $Z_0$ onto sliding orbits of $Z$ preserving tangential singularities. Now, we follow the same idea of the proof of Proposition \ref{nosigma} to construct a neighborhood $U$ of $\s$ and a homeomorphism $h:U\rightarrow U$ such that $h|_{\s}=h_s$ satisfying Definition \ref{semiequivalencedef} ($N=\s$), and we conclude that $Z_0\in\Or_{\s}$.

Finally, assume that $U_0\neq \s$, which means that $S_{Z_0}\neq \emptyset$. Hence, $\partial U_0$ is a reunion of circles of $S_{X_0}$ and $S_{Y_0}$ intersecting themselves transversally at fold-fold points of $Z_0$ (Proposition \ref{SZ}). 

From Proposition \ref{blocks}, there exist neighborhoods $\V_0$ of $Z_0$ in $\Or$ and $V_0$ (compact) of $U_0$ in $\s$, such that $\partial V_{0}\subset\s^{c}(Z_0)$, and each $Z\in\V_0$ has a unique $\s$-block $U$ in $V_0$ with the same characteristics of $U_0$, i.e., $Z\in\V$ satisfies the following properties:
\begin{itemize}
	\item[(i)] $U_0$ and $U$ have the same number of cusp-regular and fold-fold points of the same type;
	\item[(ii)] There exists $\e_0>0$, such that, if $p\in\partial U$ is either a cusp-regular or fold-fold point, then there exists a unique point $p_0\in\partial U_0$ of the same type of $p$ such that $|p-p_0|<\e_0$;
	\item[(iii)] If $p_0^{1}$ and $p_0^{2}$ are points of $\partial U_0$ connected by a curve $\Gamma_0$ of either visible or invisible fold-regular points contained in $\partial U_0$, then there exists points $p^1$ and $p^2$ of $\partial U$ of the same type of $p_0^{1}$ and $p_0^{2}$, respectively, and a unique curve $\Gamma\subset\partial U$ of fold-regular points of the same type of $\Gamma_0$, such that $d(\Gamma,\Gamma_0)<\e_0$ ($d$ denotes the Hausdorff distance). 
\end{itemize}

Notice that, it implies that the circles of $U_0$ and $U$ have the same configuration, for each $Z\in\V_0$. Let $Z\in\V$. We shall construct a semi-local equivalence between $Z_0$ and $Z$ at $U_0$. 

\subsection{Local Description of the Invariant Manifolds of Elementary Tangential Singularities}\label{local}

Firstly, we use Vishik's Normal Form Theorem \ref{Vishik} to distinguish local invariant manifolds of elementary tangential singularities.

\subsubsection{Fold-Regular} \label{fold-reg}

Let $p_0$ be a fold-regular point of $Z_0$ in $\partial U_0$, and without loss of generality, assume that $p_0\in S_{X_0}$.

From Theorem \ref{Vishik}, there exists a diffeomorphism $\Psi:V_{p_0}\rightarrow R_{p_0}$ (denote the coordinate functions of $\Psi$ by $(x_{\Psi},y_{\Psi},z_{\Psi})$) such that $\Psi(p_0)=\vec{0}$, $V_{p_0}$ is a compact neighborhood of $p_0$ in $M$, $R_{p_0}=[-l(p_0),l(p_0)]\times[-H(p_0),H(p_0)]^2\subset\rn{3}$, for some $H(p_0),l(p_0)>0$, $f(x_{\Psi},y_{\Psi},z_{\Psi})=z_{\Psi}$ and the orbit through a point $p\in V_{p_0}$ of $X_0$ is carried into the orbit of $\widetilde{X_0}(x,y,z)= (0,1,\xi y)$ passing through $\Psi(p)$, where $\xi$= sgn($X_0^{2}f(p_0)$). 

Notice that $l(p_0)$ can be taken small enough such that each $Y_{0}f(q)\neq 0$, for each $q\in V_{p_0}$. The flow of $\widetilde{X_0}$ is given by: 		
\begin{equation}
\varphi_{\widetilde{X_0}}(t;x_{0},y_{0},z_{0})= \left(x_{0},\ t+y_{0},\ \xi\dfrac{(t+y_{0})^{2}}{2} + z_{0}-\xi \dfrac{y_{0}^{2}}{2}\right),
\end{equation}

Firstly, consider $\xi>0$ and notice that $\widetilde{X_0}$ is transverse to each side of $R_{p_0}\cap M^{+}$, except for the fold-regular lines $L(\alpha)=\{(x,y,z); |x|\leq l(p_0), y=0, z=\alpha\}$, $\alpha=0,H(p_0)$, and sgn$(\widetilde{X_0}^2f(p))=$sgn$(\widetilde{X_0}^2f(p_0))$, for each $p\in L(0)\cup L(H(p_0))$.

Now, the trajectory of $\widetilde{X_0}$ through $(x,0,0)\in L(0)$ reaches $z=H(p_0)$ at the points $x^{\pm}=\left(x,\pm \sqrt{2H(p_0)}, H(p_0)\right)$, when $t=\pm \sqrt{2H(p_0)}$, respectively. Choosing $H(p_{0})$ sufficiently small, it follows that $x^{\pm}\in R_{p_0}$, for every $x\in L(0)$.

In addition, using the Flow Box Theorem, we can reduce $H(p_0)$ and find a diffeomorphism $\varPhi: R_{p_0}\cap M^{-}\rightarrow R_{p_0}\cap M^{-}$ such that $\varPhi|_{z=0}=Id$ and the orbit of $Y_0$ through $p\in R_{p_0}\cap M^{-}$ is carried onto the orbit of $\widetilde{Y_{0}}(x,y,z)=(0,0,1)$ passing through $\varPhi(p)$. Considering the homeomorphism $\Theta: V_{p_0}\rightarrow R_{p_0}$ (which is a piecewise diffeomorphism) given by:
$$\Theta(p)=\left\{\begin{array}{lcl}
\Psi(p) & \textrm{ if } & p\in M^+,\\
\Psi\circ \varPhi(p) & \textrm{ if } & p\in M^-,\\		 
\end{array}\right. $$
we define the \textbf{local 2-dimensional invariant manifold of $X_0$ at $p_0$} as:
\begin{equation}
W_{p_0}=\Theta^{-1}\left( W_{p_0}^+ \cup W_{p_0}^-\right),
\end{equation}
where $W_{p_0}^+=\left\{(x,t,t^{2}/2); |x|\leq l(p_0),\ |t|\leq\sqrt{2H(p_0)})\right\}$, and $W_{p_0}^-=\{(x,0,t);  |x|\leq l(p_0),\ -H(p_0)\leq t\leq 0 \}$.

\begin{rem}
	Notice that $\Theta$ preserves the foliation generated by $Z_0$ and $\widetilde{Z_{0}}=(\widetilde{X_{0}},\widetilde{Y_{0}})$, nevertheless it does not preserve the orientation of the orbits. 
	
	It is worth to mention that, since $W_{p_0}$ depends exclusively on the flow of $X_0$ and $Y_0$ and on the tangential curve of $X_0$ with $\s$, it follows that the existence of $W_{p_0}$ does not depend on $\Theta$. Nevertheless, $\Theta$ provides a complete characterization of $W_{p_0}$.
\end{rem}	

In this case, the foliation $\mathcal{F}$ generated by $Z_0$ in $\textrm{int}(V_{p_0})$ can be characterized in the following way. Let $N_p$ be the normal vector of $\s$ at $p$ and consider the $2$-dimensional manifold $\Lambda=\{p+\lambda N_p;\ \lambda>0, p\in S_{Z_0}\}\cap V_{p_0}$. Therefore, each leaf of $\mathcal{F}\setminus W_{p_0}$ is either a piecewise-smooth curve which passes transversally through a point of $\s$ and intersects $\partial V_{p_0}$ transversally in two points (one in $M^+$ and another in $M^-$) or it is a smooth curve which passes transversally through a point in $\Lambda$ and  intersects $\partial V_{p_0}$ in two points of $M^{+}$.		

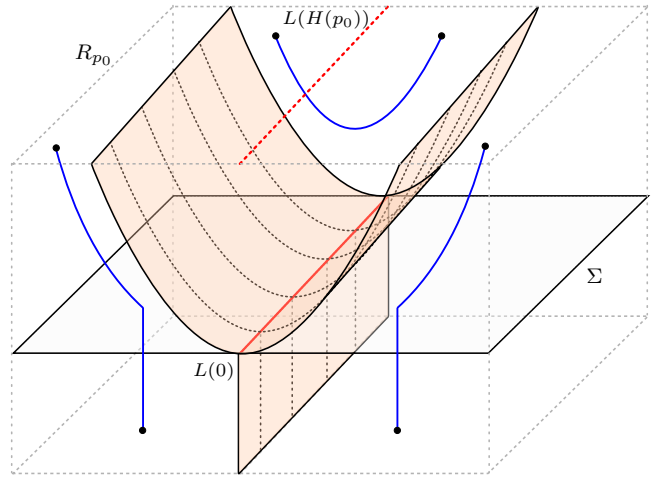
\begin{figure}[H]
	\centering
	\bigskip
	\begin{overpic}[width=8.5cm]{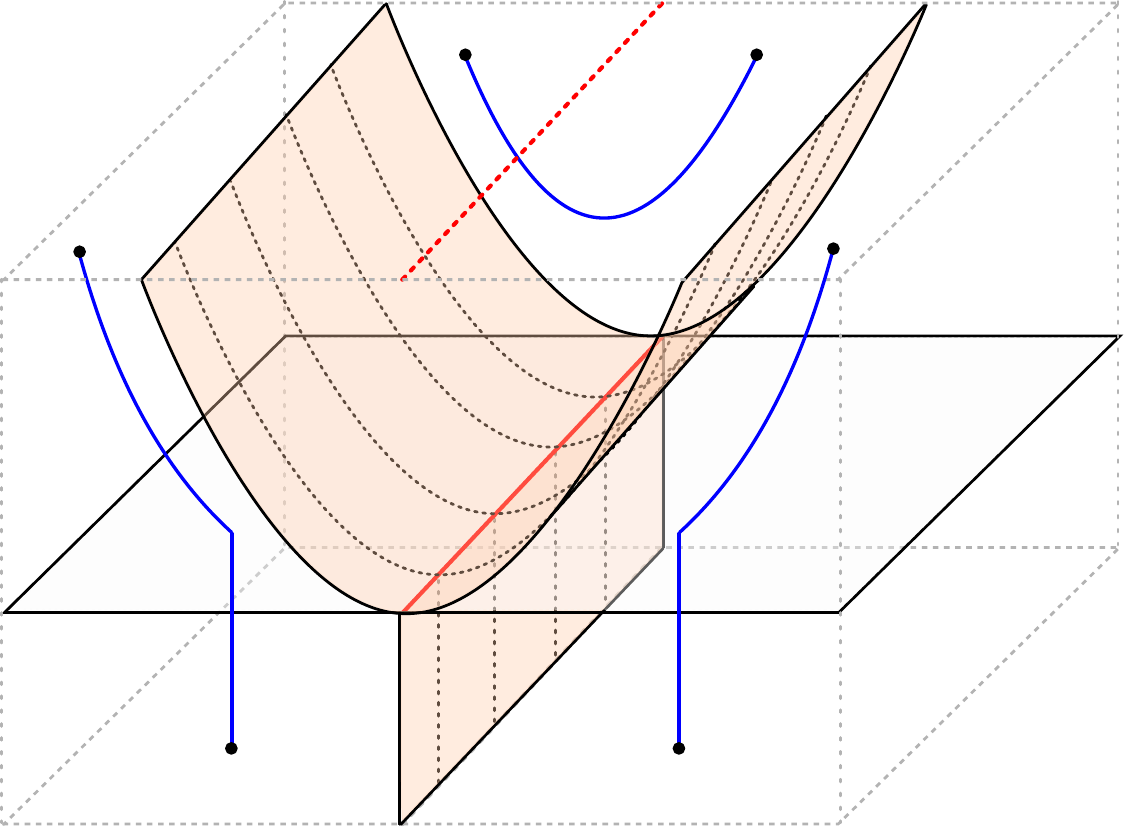}
		\put(10,65){{\footnotesize $R_{p_0}$}}		
		\put(90,30){{\footnotesize $\s$}}			
		\put(28.5,15.5){{\scriptsize $L(0)$}}	
		\put(42.5,70.5){{\scriptsize $L(H(p_0))$}}										
	\end{overpic}
	\bigskip
	\caption{The local invariant manifold $W_{p_0}$ for a visible fold-regular point.  }
\end{figure} 

Now, if $\xi<0$, we can define the same objects by changing the roles of $L(0)$ and $L(H(p_0))$. Nevertheless, we consider $W^{-}_{p_0}=\{(x,y,t);  |x|\leq l(p_0),\ -H(p_0)\leq t\leq 0,\ y=0 \textrm{ or } y=\pm\sqrt{2H(p_0)} \}$. Also, in this case, the leaves of the foliation passing through $\Lambda$ are also piecewise-smooth, intersect $\partial V_{p_0}$ at two points of $M^-$ and intersect $\s$ at two points, which are in opposite sides in relation to $S_{Z_0}$).

\begin{figure}[H]
	\centering
	\bigskip
	\begin{overpic}[width=8.5cm]{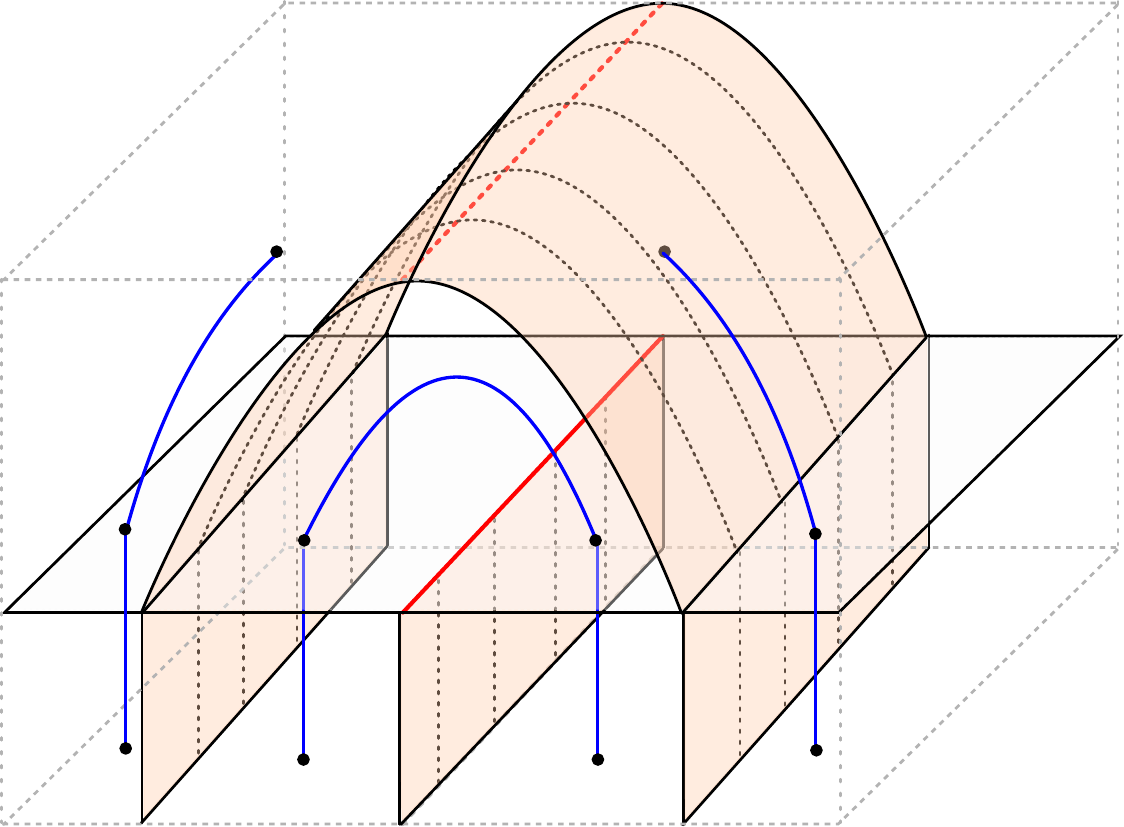}
		\put(10,65){{\footnotesize $R_{p_0}$}}		
		\put(90,30){{\footnotesize $\s$}}			
		\put(28.5,15.5){{\scriptsize $L(0)$}}	
		\put(44,74){{\scriptsize $L(H(p_0))$}}							
	\end{overpic}
	\bigskip
	\caption{The local invariant manifold $W_{p_0}$ for an invisible fold-regular point. }
\end{figure}
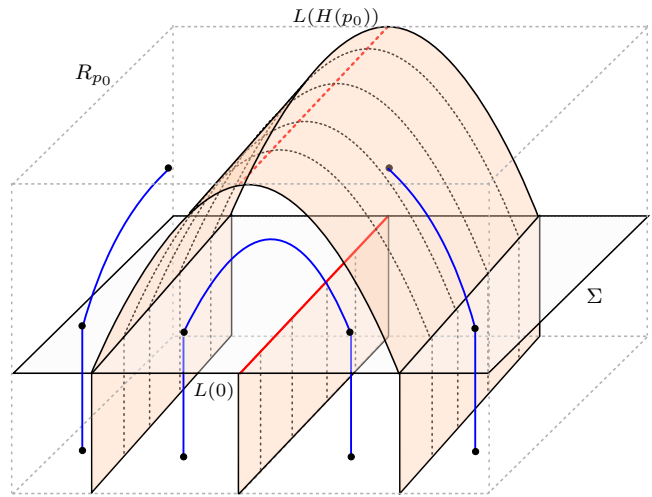

\begin{rem}\label{remff}
	Notice that if $X_0\in\Xr$ is a smooth vector field defined in a $\overline{\s^{+}}$ which has a fold point at $p_0$, then we can construct the local invariant manifold of $X_0$ at $p_0$ in the same way.	
\end{rem}	

\subsubsection{Cusp-Regular}

Let $p_0$ be a cusp-regular point of $Z_0$ in $\partial U_0$, and assume that $p_0\in S_{X_0}$.

Following the same arguments and notation of Section \ref{fold-reg}, there exists a diffeomorphism $\Psi(p_0)=\vec{0}$, $V_{p_0}$ such that $f(x_{\Psi},y_{\Psi},z_{\Psi})=z_{\Psi}$ and the orbit through a point $p\in V_{p_0}$ of $X_0$ is carried into the orbit of $\widetilde{X_0}(x,y,z)= (1,0,\xi x^2+y)$ passing through $\Psi(p)$, where $\xi$= sgn($X_0^{3}f(p_0)$) in this case. 

We only consider $\xi=1$, since the invariant manifolds obtained when $\xi=-1$ are completely analogous.  The flow of $\widetilde{X_0}$ is given by: 		
\begin{equation}
\varphi_{\widetilde{X_0}}(t;x_{0},y_{0},z_{0})= \left(t+x_{0},\ y_{0},\ \dfrac{(t+x_{0})^{3}}{3} + y_0 t + z_{0}- \dfrac{x_{0}^{3}}{3}\right),
\end{equation}

Now, $\widetilde{X_0}$ is transverse to each side of $R_{p_0}\cap M^{+}$, except for the curves $L(\alpha)=\{(x,y,z); |x|\leq \sqrt{l(p_0)}, y=-x^2, z=\alpha\}$, $\alpha=0,H(p_0)$. In addition, sgn$(\widetilde{X_0}^2f((x,y,z)))=$sgn$(x)$, for each $(x,y,z)\in L(0)\cup L(H(p_0))$ with $x\neq 0$, and $(0,0,\alpha)$ are cusp points of $\widetilde{X_0}$ such that sgn$(\widetilde{X_0}^3f((0,0,\alpha)))= 1$.

Let, $p_x=(x,-x^{2},0)$ be a visible fold-regular point of $L(0)$ ($x>0$), then the orbit through $p_x$ is given by $\gamma_{x}(t)=(t+x, -x^{2}, (t+x)^{3}/3-x^{2}t-x^{3}/3)$. In particular, if $h_0=3/4H(p_0)>0$, then $\gamma_{h_0}$ intersects $H(p_0)$ at $(-h_0, -h_0^{2}, H(p_0))\in L(H(p_0))$ (take $H(p_0)$ such that $h_0<\sqrt{l(p_0)}$). Also, notice that $\gamma_{h_0}$ is the orbit through the visible fold-regular point $(-h_0, -h_0^{2}, H(p_0))$.

In addition, $\gamma_{h}$ is tangent to the planes $z=0$ and $z=4/3{h}^{3}=\delta_h$ and it is contained in the plane $y=-h^2$. Also, $\gamma_{h}$ intersects $z=0$ at the points $({h},-{h}^2,0)$ and $P_{h}=(-2{h},-{h}^2,0)$, and it intersects $z=\delta_{h}$ at $Q_h=(-{h},-{h}^2,\delta_{h})$ and $(2{h},-{h}^2,\delta_{h})$. Notice that, $P_{h},Q_{h}\rightarrow \vec{0}$ and $\delta_{h}\rightarrow 0$ when ${h}\rightarrow 0$. Therefore, it presents the behavior presented in Figure \ref{behavior_cusp}. Consider $W_{p_0}^{+}(0)=\{\gamma_x(t);\ 0\leq x\leq\sqrt{l(p_0)},\ t\in I_{x}=[T_{-}(x), T_{+}(x)]\}$, where $I_x$ is the maximal interval such that $\gamma_x(I_x)\subset R_{p_0}$. 

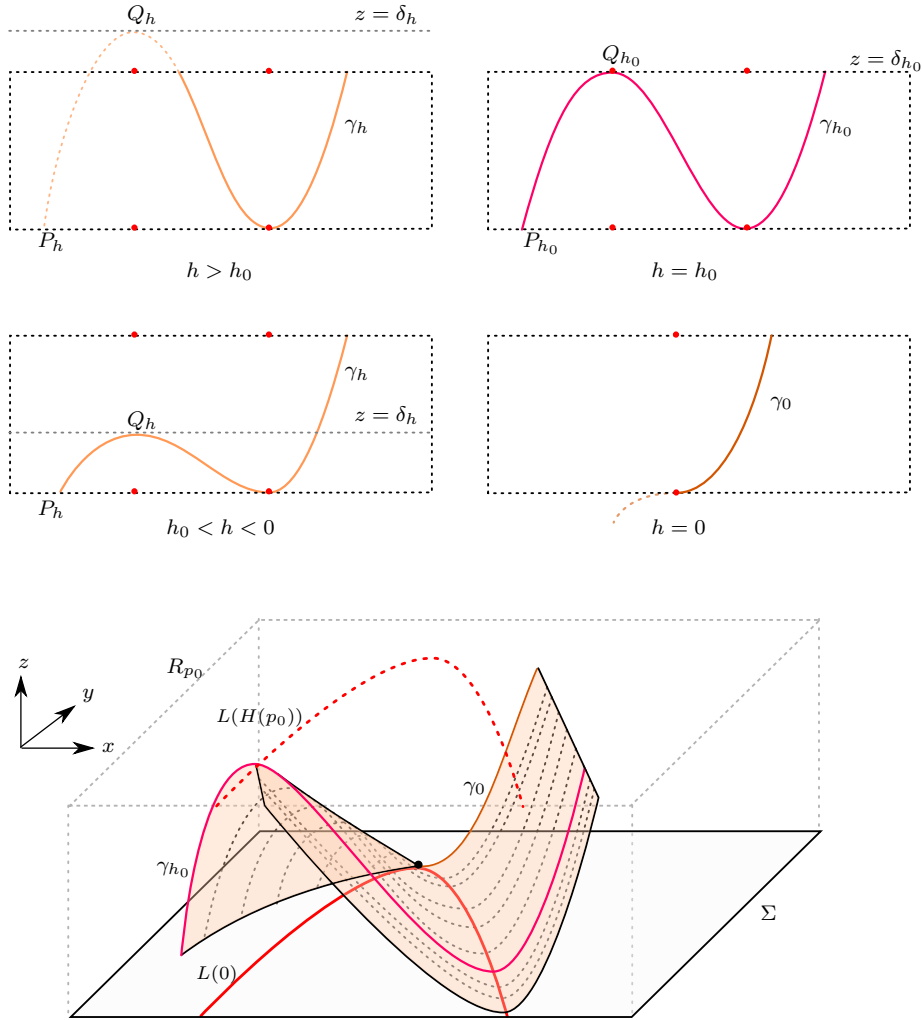
\begin{figure}[H]
	\centering
	\bigskip
	\begin{overpic}[width=12cm]{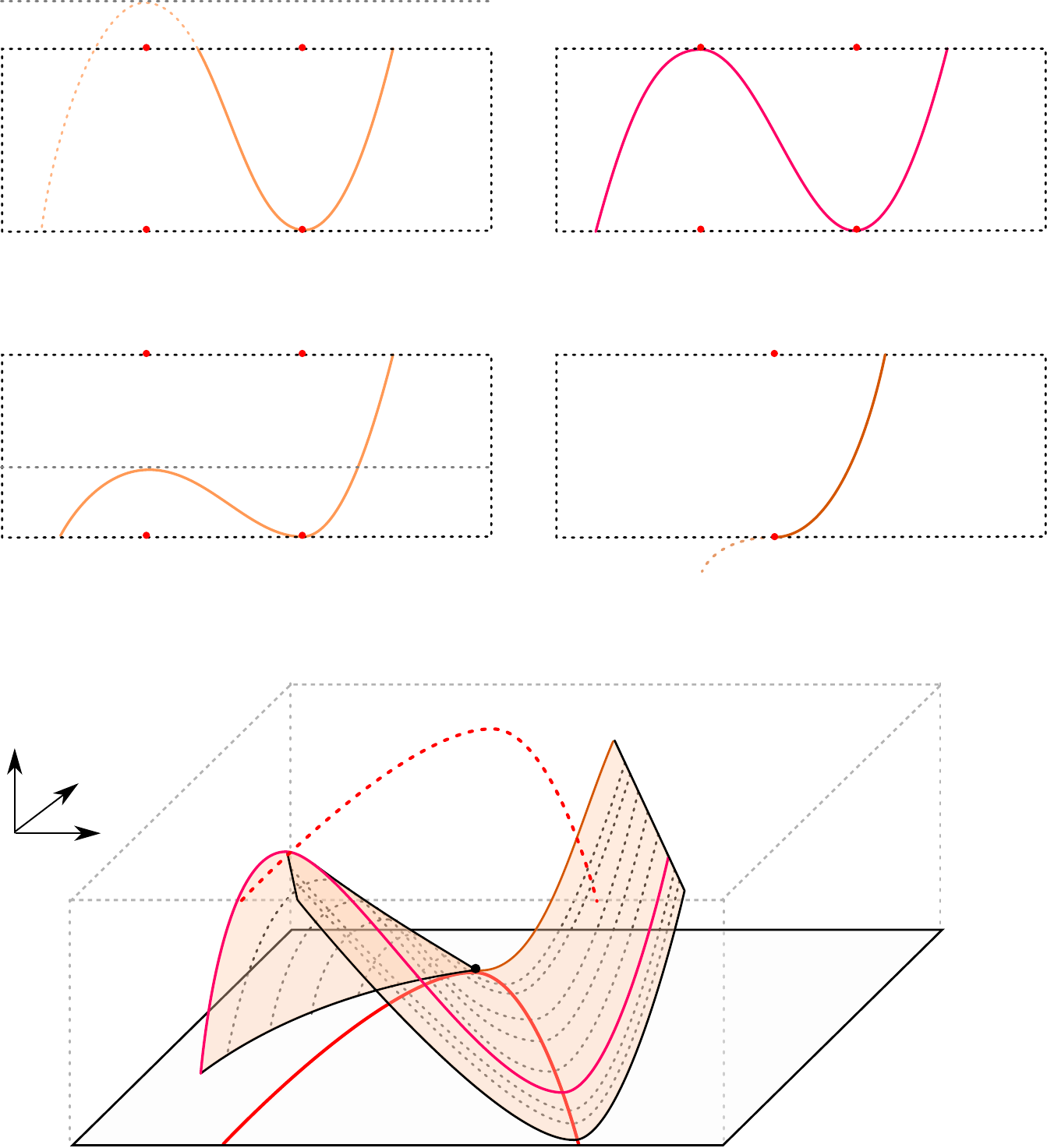}
		\put(18,75){{\footnotesize $h>h_0$}}	
		\put(3,78){{\footnotesize $P_h$}}	
		\put(12,101){{\footnotesize $Q_h$}}			
		\put(34,90){{\footnotesize $\gamma_h$}}		
		\put(52,78){{\footnotesize $P_{h_0}$}}	
		\put(60,97){{\footnotesize $Q_{h_0}$}}			
		\put(82,90){{\footnotesize $\gamma_{h_0}$}}										
		\put(77,62){{\footnotesize $\gamma_{0}$}}		
		\put(46,23){{\footnotesize $\gamma_{0}$}}																
		\put(3,51){{\footnotesize $P_h$}}	
		\put(12,60){{\footnotesize $Q_h$}}			
		\put(34,65){{\footnotesize $\gamma_h$}}																																				
		\put(65,75){{\footnotesize $h=h_0$}}		
		\put(16,49){{\footnotesize $h_0<h<0$}}			
		\put(65,49){{\footnotesize $h=0$}}		
		\put(35,101){\footnotesize $z=\delta_{h}$}			
		\put(85,96.5){\footnotesize $z=\delta_{h_0}$}		
		\put(35,60.5){\footnotesize $z=\delta_{h}$}																																	
		\put(16,35){{\footnotesize $R_{p_0}$}}			
		\put(21,30){{\scriptsize $L(H(p_0))$}}				
		\put(19,4){{\scriptsize $L(0)$}}
		\put(76,10){{\footnotesize $\s$}}			
		\put(15,15){{\footnotesize $\gamma_{h_0}$}}	
		\put(9.5,27){{\footnotesize $x$}}		
		\put(7.5,32.5){{\footnotesize $y$}}
		\put(1,35.5){{\footnotesize $z$}}																																																																				
	\end{overpic}
	\bigskip
	\caption{The local invariant manifold $W^{+}_{p_0}(0)$ for a cusp point and its description in the slice $y=-h^2$ of $R_{p_0}$.  } \label{behavior_cusp}
\end{figure} 

Changing the roles of $z=0$ and $z=H(p_0)$, we can define an analogous $2$-manifold $W_{p_0}^{+}(H(p_0))=\{\varphi_{\widetilde{X_0}}(t;x,-x^2,H(p_0));\ -\sqrt{l(p_0)}\leq x\leq 0,\ t\in I_{x}=[T_{-}(x), T_{+}(x)]\}$, where $I_x$ is the maximal interval such that $\varphi_{\widetilde{X_0}}(I_x;x,-x^2,H(p_0))\subset R_{p_0}$. 

\begin{figure}[H]
	\centering
	\bigskip
	\begin{overpic}[width=12cm]{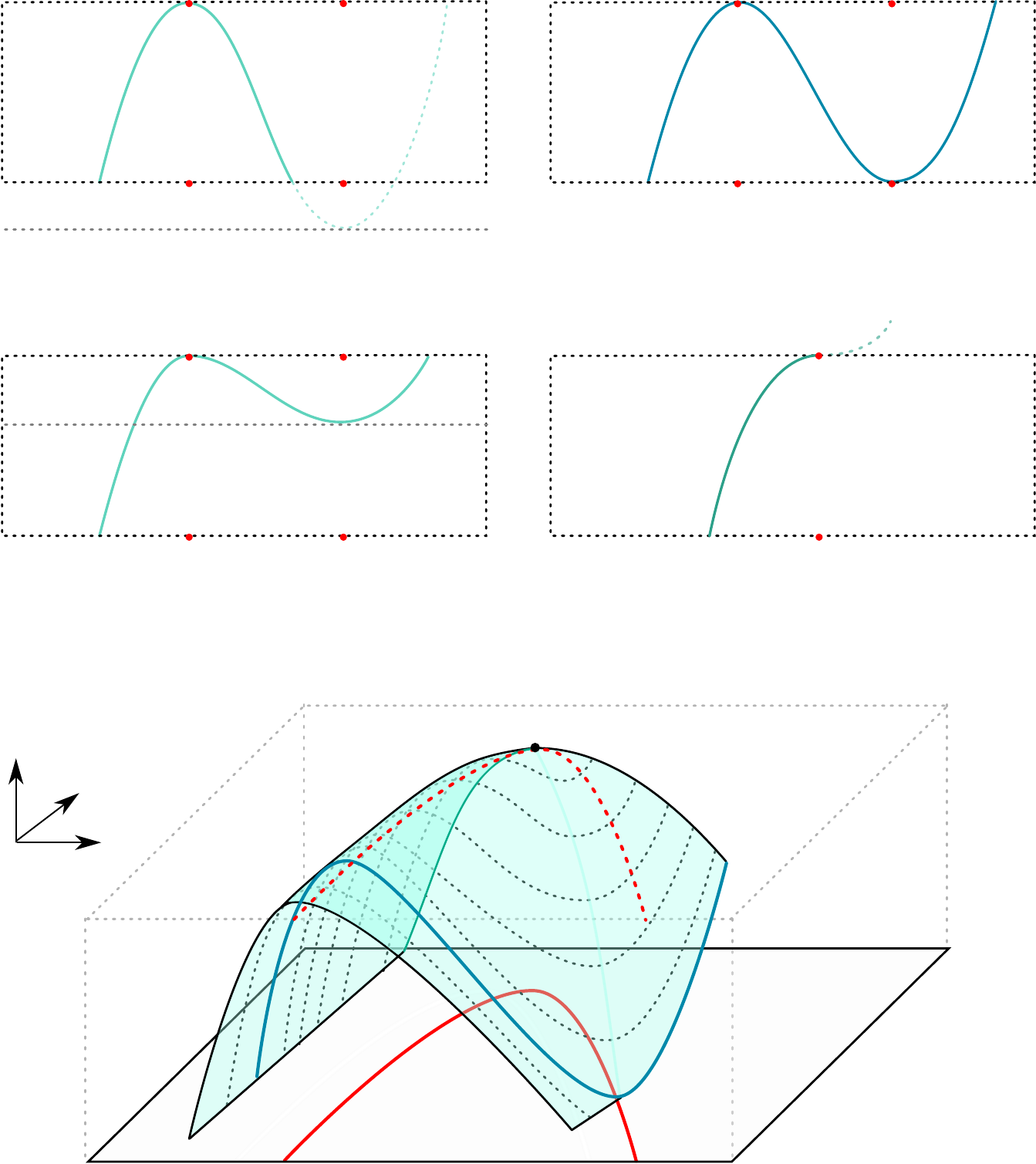}
		\put(9.5,27){{\footnotesize $x$}}		
		\put(7.5,32.5){{\footnotesize $y$}}
		\put(1,35.5){{\footnotesize $z$}}
		\put(18,75){{\footnotesize $h<-h_0$}}	
		\put(29,77){{\footnotesize $\widetilde{P_h}$}}	
		\put(15,101){{\footnotesize $\widetilde{Q_h}$}}			
		\put(7,90){{\footnotesize $\widetilde{\gamma_h}$}}		
		\put(76,81){{\footnotesize $P_{h_0}$}}	
		\put(62,101){{\footnotesize $Q_{h_0}$}}			
		\put(52,90){{\footnotesize $\gamma_{h_0}$}}										
		\put(60,62){{\footnotesize $\widetilde{\gamma_{0}}$}}		
		\put(40,29){{\footnotesize $\widetilde{\gamma_{0}}$}}																
		\put(28,61){{\footnotesize $\widetilde{P_h}$}}	
		\put(15,70.5){{\footnotesize $\widetilde{Q_h}$}}			
		\put(9,65){{\footnotesize $\widetilde{\gamma_h}$}}																																				
		\put(65,75){{\footnotesize $h=-h_0$}}		
		\put(16,49){{\footnotesize $-h_0<h<0$}}			
		\put(65,49){{\footnotesize $h=0$}}		
		\put(35,81){\footnotesize $z=\widetilde{\delta_{h}}$}			
		\put(82,81.5){\footnotesize $z=\widetilde{\delta_{h_0}}=0$}		
		\put(34,60.5){\footnotesize $z=\widetilde{\delta_{h}}$}																																	
		\put(16,35){{\footnotesize $R_{p_0}$}}			
		\put(26,30){{\scriptsize $L(H(p_0))$}}				
		\put(24,4){{\scriptsize $L(0)$}}
		\put(76,10){{\footnotesize $\s$}}			
		\put(60,15){{\footnotesize $\gamma_{h_0}$}}												
	\end{overpic}
	\bigskip
	\caption{The local invariant manifold $W^{+}_{p_0}(H(p_0))$ for a cusp point and its description in the slice $y=-h^2$ of $R_{p_0}$. We denote by $\widetilde{P_h},\ \widetilde{Q_h}$ and  $\widetilde{\gamma_h}$ the elements analogous to $P_h,\ Q_h$ and $\gamma_h$ in Figure \ref{behavior_cusp}, respectively.   }\label{behavior_H}
\end{figure} 	

Notice that, $W_{p_0}^{+}(0)$ and $W_{p_0}^{+}(H(p_0))$ intersect themselves transversally at the curve $\gamma_{h_0}$. Let $W_{p_0}^+=W_{p_0}^{+}(0)\cup W_{p_0}^{+}(H(p_0))$ and consider $S=(W_{p_0}\cup S_{\widetilde{X_{0}}})\cap\{z=0\}$, hence the invariant manifold $W^{-}_{p_0}$ is constructed as in the fold-regular case, but here we take it as the image of the flow of $\widetilde{Y_{0}}$ through $S$, for $-H(p_0)\leq t\leq 0$.

In this case, the local 2-dimensional invariant manifold of $Z_0$ at $p_0$ is given by $W_{p_0}=\Theta^{-1}\left( W_{p_0}^+ \cup W_{p_0}^-\right)$.

In Figure \ref{foliation}, the foliation of $Z_0$ in $R_{p_0}$ is described. For simplicity, we characterize it on each plane $y=k$, where $-l(p_0)\leq k\leq l(p_0)$. Notice that $R_{p_0}$ is partitioned in regions where the behavior is of type either transversal or visible fold-regular or invisible fold-regular, and the formal description of this regions can be found in Section \ref{fold-reg}.

\begin{figure}[H]
	\centering
	\bigskip
	\begin{overpic}[width=12cm]{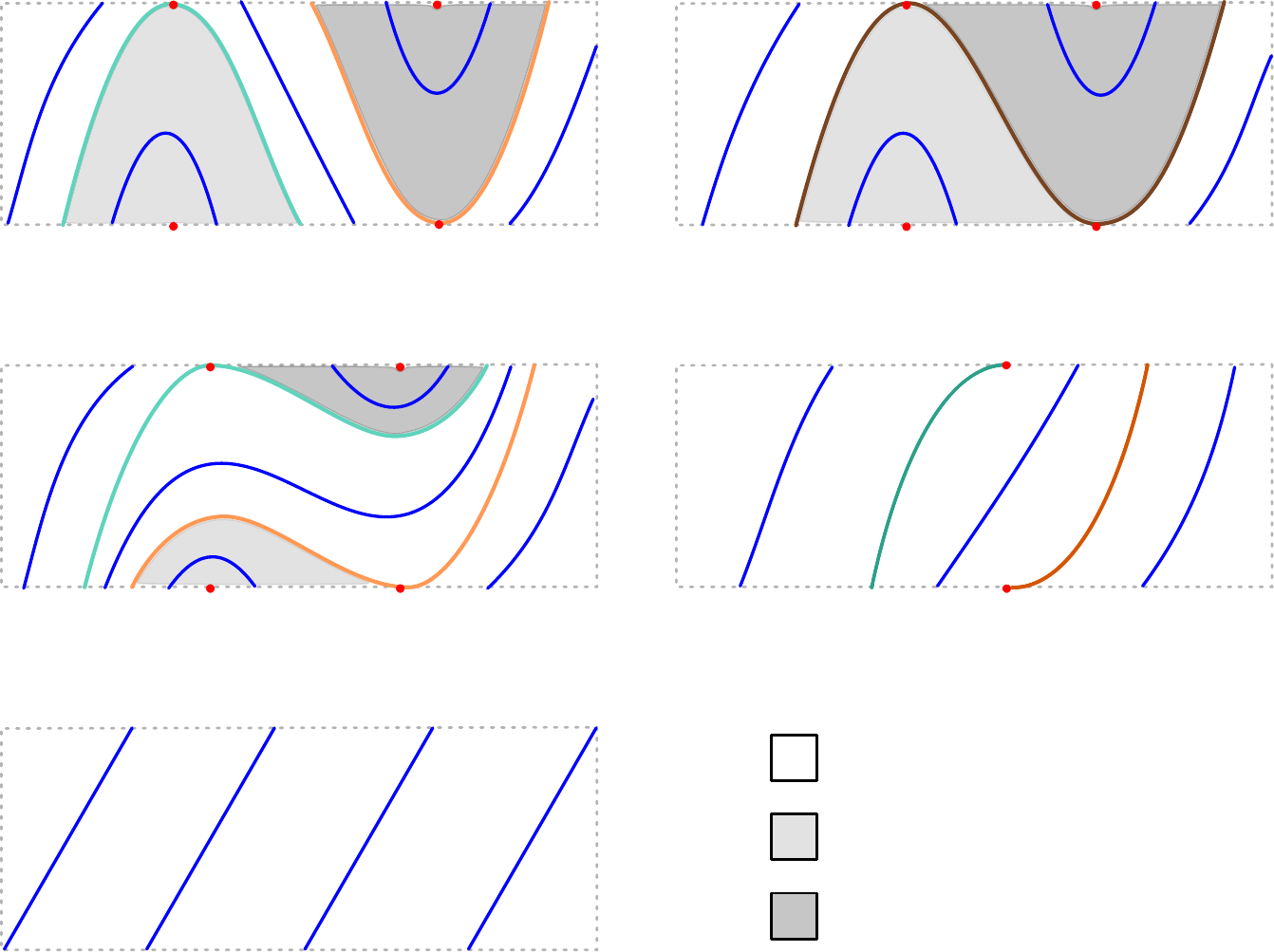}
		
		\put(19,51.5){\footnotesize $y<-h_0^2$}
		\put(5,75.7){\tiny $W_{p_0}^+(H(p_0))$}
		\put(37,55){\tiny $W_{p_0}^+(0)$}			
		\put(75,51.5){\footnotesize $y=-h_0^2$}
		\put(57,55){\tiny $W_{p_0}^+(H(p_0))=W_{p_0}^+(0)$}
		\put(37,55){\tiny $W_{p_0}^+(0)$}		
		\put(17,22){\footnotesize $-h_0^2<y<0$}
		\put(12,47.3){\tiny $W_{p_0}^+(H(p_0))$}
		\put(29,26){\tiny $W_{p_0}^+(0)$}					
		\put(75,22){\footnotesize $y=0$}
		\put(71,47.3){\tiny $W_{p_0}^+(H(p_0))$}
		\put(78,26){\tiny $W_{p_0}^+(0)$}	
		\put(21,-4){\footnotesize $y>0$}
		\put(66,14.5){\footnotesize Regular Sector}	
		\put(66,8.2){\footnotesize Invisible Fold Sector}	
		\put(66,2){\footnotesize Visible Fold Sector}														
	\end{overpic}
	\bigskip
	\caption{Foliation generated by $Z_0$ in the slices $y=k$ ($k$ is a constant) of the neighborhood $R_{p_0}\cap\{z\geq 0\}$.}\label{foliation}
\end{figure}
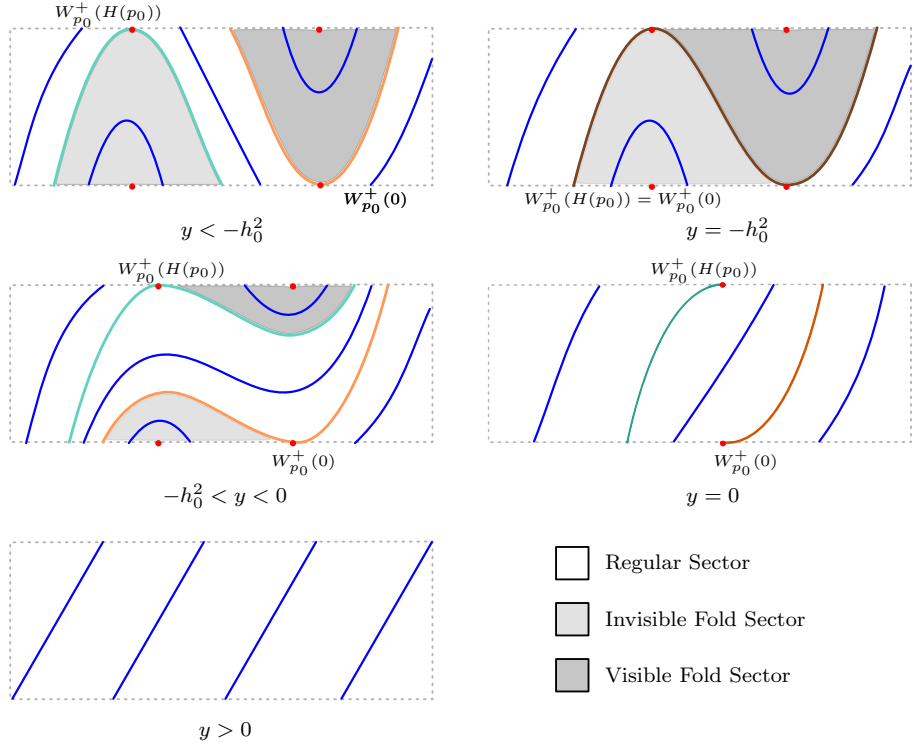

\subsubsection{Fold-Fold} \label{foldfold}		

If $p_0$ is a fold-fold point of $Z_0$, then we can construct the invariant manifolds of $X_0$ and $Y_0$ for the fold-lines of $X_0$ and $Y_0$, respectively, by following Section \ref{fold-reg} (see Remark \ref{remff}). The resultant manifolds can be seen in Figure \ref{types} for each type of fold-fold singularity.

In addition, if $p_0$ satisfies the conditions of local structural stability at $p_0$, then the results of \cite{GT} allows us to consider that, there exists a homeomorphism $h_{p_0}: V_{p_0}\rightarrow V_{p_0}$ which carries orbits of $Z_0$ onto orbits of $Z$.

Also, notice that, all the trajectories outside the local invariant manifolds intersect $\partial V_{p}$ transversally, and if we consider a neighborhood $V$ of $\s$ in $M$ sufficiently small, an orbit contained in $V$ can intersect $\s$ more than one time only inside neighborhoods $V_{p}$ of elliptic fold-fold points. 

Therefore, there exist only local first return maps in $V$, and since $h_{p_0}$ is a local equivalence between $Z_{0}$ and $Z$ at $p_0$, we can extend it into a semi-local equivalence between $Z_0$ and $Z$ at $\s$. Hence, local first returns are not an obstruction to have semi-local structural stability at $\s$. 

\subsection{Existence of the Invariant Manifold of a Tangency Set}
Now, we must show that the local invariant manifolds of elementary tangential singularities give rise to global invariant manifolds of $S_{Z_0}$ defined in a neighborhood $\Lambda$ of the whole $\s$ in $M$. The process mainly consists on the use of the compactness of $\s$ to concatenate the local manifolds in a smooth way.    

\begin{rem}
	The term global invariant manifold is used to emphasize that it is defined in a neighborhood of the whole $\s$. In fact, it is a collection of local invariant manifolds.
\end{rem}

Let $N_p$ be the normal vector of $\s$ at $p$ pointing toward $\s$. Consider the following \textbf{$\lambda$-lamination of $\s$}:
\begin{equation}
\s_{\lambda}=\{p+\lambda N_p;\ p\in\s\},
\end{equation}
where $\lambda\in\R$.

Let $p\in\s$, since $Z_0\in\Xi_0$, $p$ is either a regular-regular or a fold-regular or a cusp-regular or a fold-fold point of $Z_0$. Hence, let $V_p$ be a compact neighborhood of $p$ in $M$ such that:
\begin{itemize}
	\item[(i)]	If $p$ is regular-regular, then each $q\in\s\cap V_{p}$ is a regular-regular point of $Z_0$ and $X_0,Y_0$ are transverse to $\partial V_p$;
	\item[(ii)] If $p$ is an elementary tangential singularity, consider the neighborhood $V_{p}$ given in Section \ref{local}. 
\end{itemize}

From compactness of $\s$, we can find a finite subcoverture $V=V_{p_1}\cup\cdots V_{p_n}$ of $\s$. Therefore, there exists $\lambda^{*}>0$ such that $\Lambda=\cup_{\lambda\in[-\lambda^*,\lambda^*]} \s_{\lambda}$ is contained in $V$.

Notice that, for each $p\in S_{Z_0}$, the laminations $\s_{\pm\lambda^*}\cap V_{p}$ correspond to planes $z=\pm k$ in the neighborhood $R_{p}$, for some $k>0$. For simplicity, we assume $H(p)=k$. 

Since $\Lambda$ is constructed by laminations of $\s$ in the direction of the normal vectors of $\s$, it follows that the same tangency set $S_{Z_0}$ persists on $\partial \Lambda$.

\begin{figure}[H]
	\centering
	\bigskip
	\begin{overpic}[width=4.5cm]{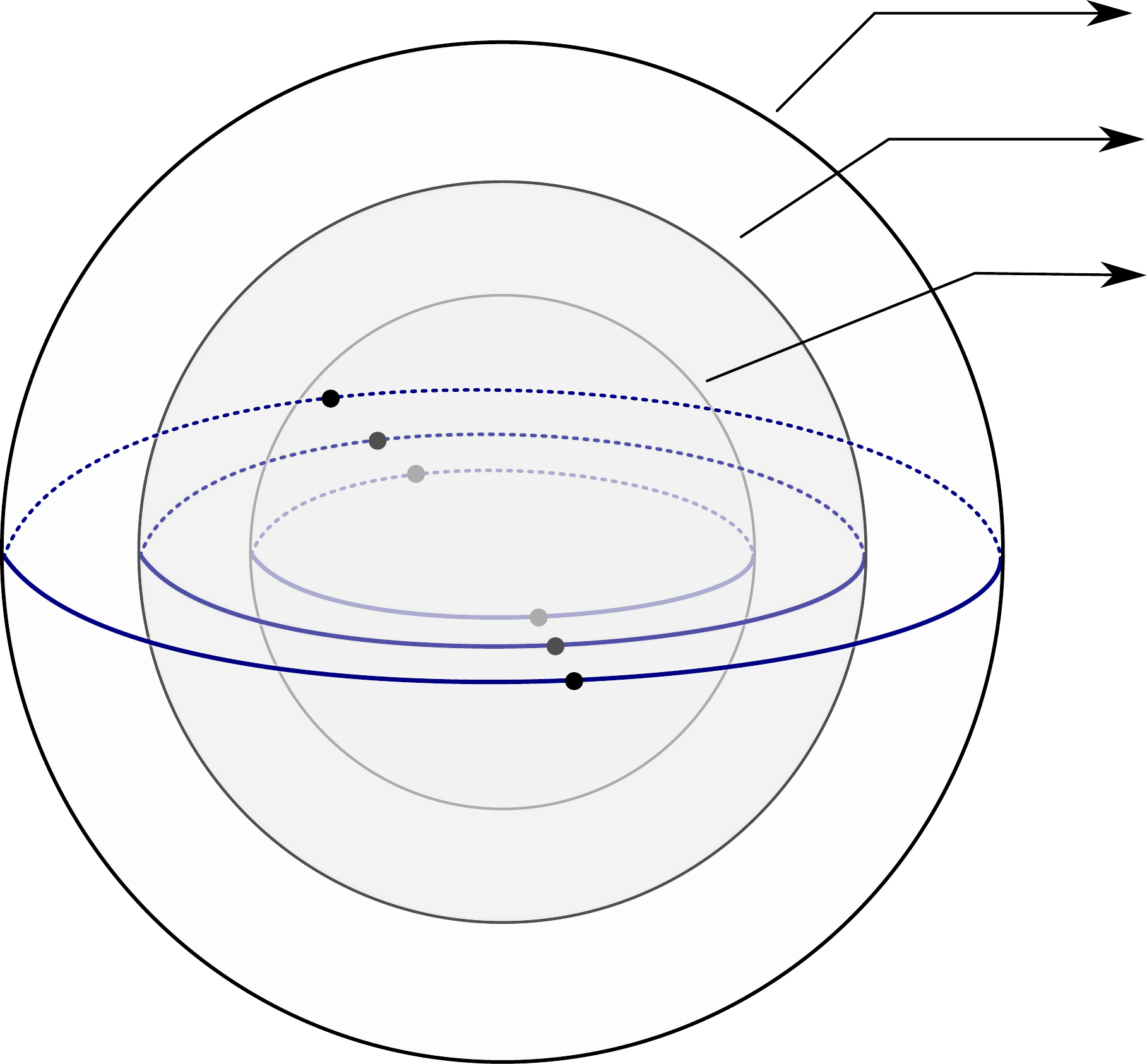}
		\put(101,67){\footnotesize $\s_{-\lambda^*}$}
		\put(101,78){\footnotesize $\s$}	
		\put(101,90){\footnotesize $\s_{\lambda^*}$}				
	\end{overpic}
	\bigskip
	\caption{ Example of neighborhood $\Lambda$, where $S_{Z_0}=\{(x,y,0);\ x^2+y^2=1\}$. The distinguished points represent cusp points of $X_0$.   }	
\end{figure}  

Recall that the local invariant manifolds of an elementary tangential singularity $p_0$ depends only on the tangency set of $\widetilde{Z_0}$ with $z=0$, $z= H(p)$ and $z=-H(p)$. Therefore, it depends intrinsically on the tangency set of $Z_0$ with $\s$,and $\s_{\pm \lambda^*}$. 

It is sufficient to prove that all the local invariant manifolds characterized in Section \ref{local} extend themselves to global invariant manifolds of $S_{Z_0}$.

In order to clarify these ideas, we explain how the local invariant manifolds originates a global invariant manifold when $S_{Z_0}=S_{X_0}$, and $S_{X_0}$ is a connected set composed by fold-regular point and two cusp-regular points.

Let $p,q$ be the cusp-regular points of $Z_0$, therefore $S_{X_0}$ is composed by two arcs $A_1$ and $A_2$ with extrema  $p,q$, such that the fold-regular points of $A_1\setminus\{p,q\}$ (resp. $A_2\setminus\{p,q\}$) are visible (resp. invisible). 

At each point $p\in A_{1}$, consider the neighborhoods $V_p$ found in Section \ref{local}. From compactness of $A_1$, a finite number of them can cover $A_{1}$, say it $V_{1}$, $\cdots$, $V_{n}$, and by connectedness, they intersect each other at least in one point.

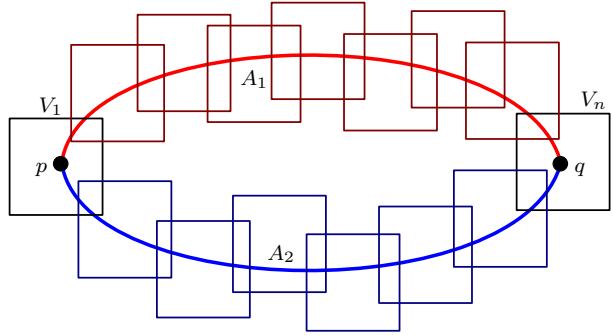
\begin{figure}[H]\label{globalvisfig}
	\centering
	\bigskip
	\begin{overpic}[width=8cm]{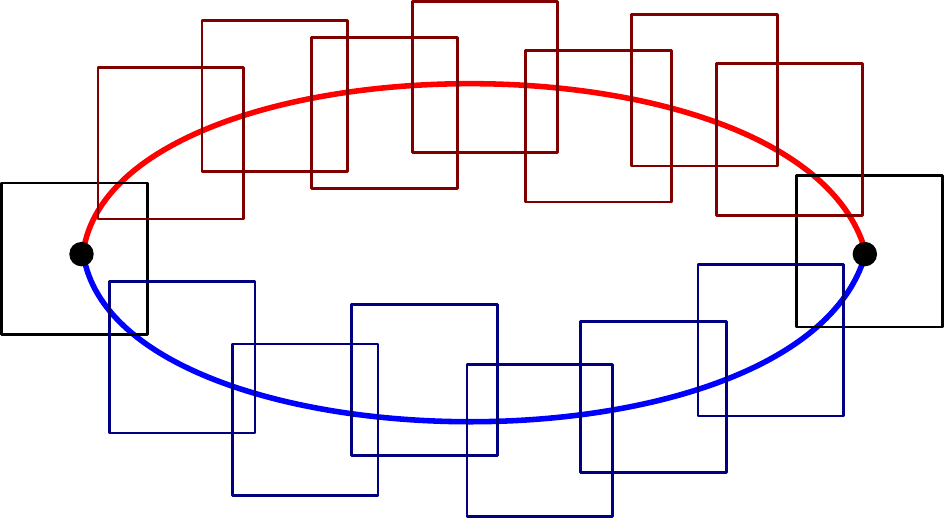}
		\put(5,36.5){\footnotesize $V_1$}
		\put(95,37.5){\footnotesize $V_n$}	
		\put(38.5,41){\footnotesize $A_1$}
		\put(43,12){\footnotesize $A_2$}	
		\put(4.5,26.5){\footnotesize $p$}
		\put(94,26.5){\footnotesize $q$}							
	\end{overpic}
	\bigskip
	\caption{Arcs $A_1$ and $A_2$ of $S_{Z_0}$.}
\end{figure} 

Now, in each $V_i$, we have that the local manifold is given by the image of  $\varphi_{X_0}(t; p)$, with $T_-^i(p)\leq t\leq T_+^i(p) $, where $T_-^i(p)<0<T_+^i(p)$ and $p\in A_1\cap V_i$. Let $q\in A_1\cap \textrm{int}( V_i\cap V_j)$, and restrict the values of $t$ to the interval with extrema $T_{+}(q)=\min\{ T_{+}^{i,j}(q) \}$ and $T_{-}(q)=\min\{ T_{-}^{i,j}(q) \}$. It is enough to reduce the heights of the neighborhoods $R_i$ to concatenate the local manifolds. Repeating this process, we extend the manifolds to the arc $A_1$ obtaining $W_1^+$.

Notice that, in the neighborhoods $V_1$ and $V_n$, we have cusp-regular points, therefore, the invariant manifold in $V_{2}\cup\cdots\cup V_{n-1}$ concatenates with the local invariant manifolds of the cusp-regular points having visible fold-regular points.

The construction of the global manifold of the arc $A_2$ is done in analogous way, but notice that, in this case, the concatenation has to be done in the visible fold-regular points at the lamination, and we obtain the global invariant manifold $W_2^+$.

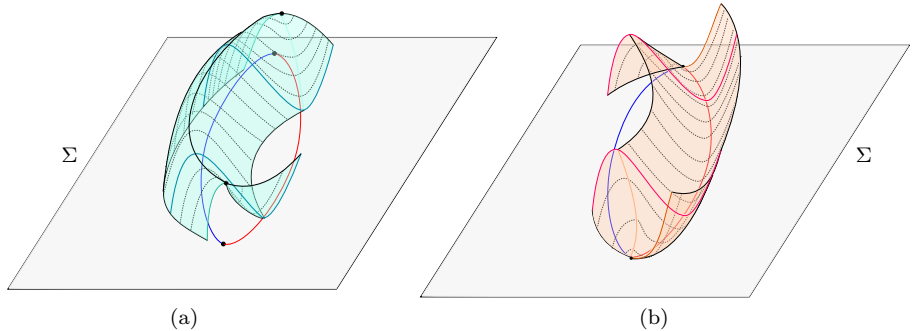
\begin{figure}[H] 
	\centering
	\bigskip
	\begin{overpic}[width=12cm]{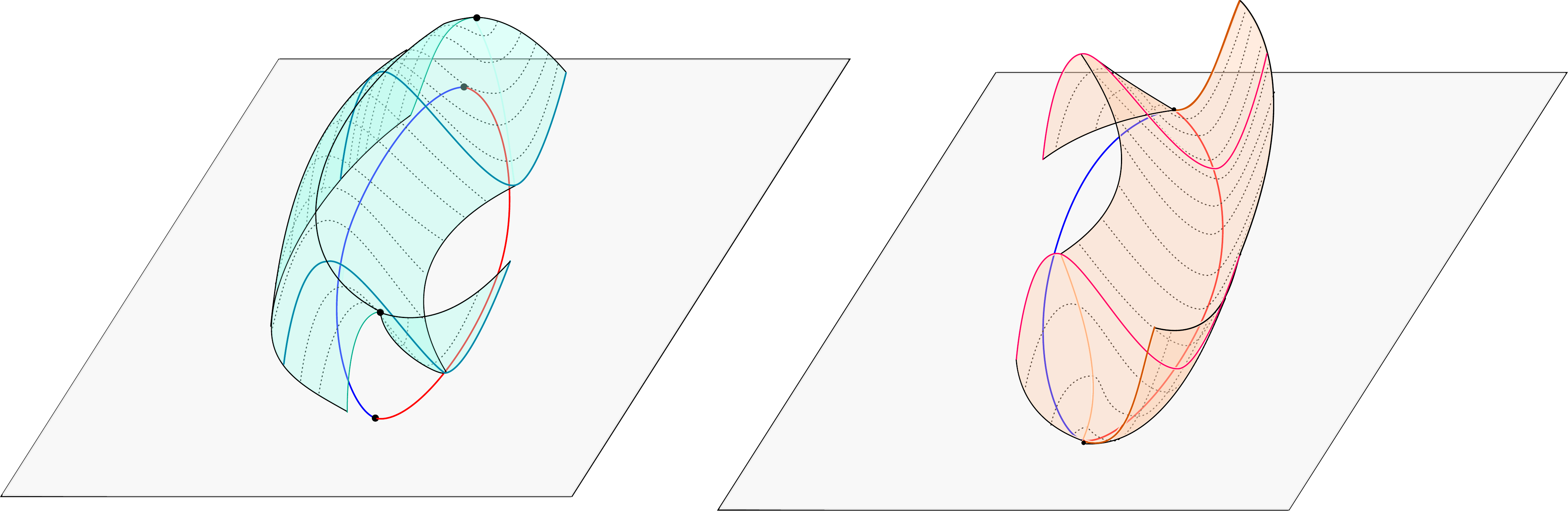}
		\put(18,-3){\footnotesize (a)}
		\put(70,-3){\footnotesize (b)}		
		\put(6,15){\footnotesize $\s$}	
		\put(94,15){\footnotesize $\s$}									
	\end{overpic}
	\bigskip
	\caption{Global manifolds $W_{2}^{+}$ (a) and $W_{1}^{+}$ (b).  }
\end{figure} 	

\subsection{Construction of the Homeomorphism}

Finally, we construct the semi-local equivalence between $Z_0$ and $Z$ at $U_0$. Firstly, define $h:U_0\rightarrow U$ by using Theorem A.

Consider the neighborhood $\Lambda$ of the previous section. From Remark \ref{remsl}, if $p_0\in\partial U_0$ is a fold-fold singularity, then we can extend $h$ into a neighborhood $W_{p_0}=V_{p_0}\cap\Lambda$, i.e. $h:W_{p_0}\rightarrow W_{p_0}$ through the remarks on Section \ref{foldfold}. 

Since $h$ carries tangential singularities of $Z_0$ onto tangential singularities of $Z$ of the same type, then $h$ carries $\partial U_0$ onto $\partial U$. Therefore,  we can use the flows of $Z_0$ and $Z$ to carry the global invariant manifolds of $U_0$ onto the global invariant manifold of $Z$. 

Recall that, outside the global manifolds, the flows of $Z_0$ and $Z$ are transversal to $\partial\Lambda$. Consider any extension of $h$ into a small compact neighborhood $W$ of $U_0 \cup U$ in $\s$.  

Let $V=\{p+\lambda N_p;\ p\in W,\ \lambda\in[-\lambda^*,\lambda^*]\}$, therefore, we can easily extend $h$ into $V$ through the flow of $Z_0$ and $Z$. In fact, the behavior of both piecewise-smooth vector fields are trivial outside global manifolds, and we can use the local foliations characterized in Section \ref{local} and transversality arguments to do this extension ( see \cite{F,GT,ST,T1} for more details).

Now, it follows from construction that $h$ carries orbits of $Z_0$ onto orbits of $Z$. Hence $Z_0$ is semi-local equivalent to $Z$ at $U_0$.

\subsubsection{Conclusion of the Proof}
We have shown that $\s_0\subset \Or_{\s}$. From Local Theory, it is easy to see that, if $Z_0\notin\s_0$ then $Z_0$ is not semi-local structurally stable at $\s$. Therefore, we have proven item $(i)$.	Items $(ii)$ and $(iii)$ of Theorem B follows directly from Corollary E of \cite{GT}.

\section{Closing Remarks and Further Directions}

In this paper, we have found necessary conditions for the structural stability in $\Or$. First of all, remark that all results stated in Section $5$ hold for vector fields having a compact oriented switching manifold $\s$ (without the simply connectedness assumption). In fact, all proofs can be done in the same way as the general case. For simplicity, we have considered $\s$ diffeomorphic to $\mathbb{S}^{2}$ just for technical reasons.

We highlight that the problem remains open for non-orientable switching manifolds. In this case, even the definition of piecewise-smooth vector fields is still not established. It certainly presents lots of mathematical challenges.

The behavior of continuous piecewise smooth vector fields is trivial around the switching manifold, nevertheless they may present a completely non-trivial dynamics from the global point of view. In light of this, the characterization of structural stability is a rather challenging problem.

Finally, the most natural extension of this work is to study what globally happens with the invariant manifolds defined in Section 8.2 in structurally stable systems. It originates applications in generic bifurcation theory. 

\section{Acknowledgments}
O. M. L. Gomide is supported by the FAPESP grants 2015/22762-5 and 2016/23716-0. M. A. Teixeira is supported by the FAPESP grant 2012/18780-0 and by the CNPq grant 300596/2009-0.

\bibliographystyle{amsplain}


\end{document}